\newif\ifisarxive
\isarxivetrue 

\ifisarxive
    \documentclass{article}
    \pdfoutput=1
    \usepackage{arxiv}
\else
    \documentclass[12pt]{article}
    \usepackage[letterpaper, margin=1in]{geometry}
\fi

\usepackage{setspace}

\usepackage{graphicx} 

\usepackage{pgfgantt}
\usepackage{tikz}
\usepackage{diagbox}

\usepackage{listings}

\usepackage{pdflscape}
\usepackage{rotating}

\lstdefinelanguage{Macaulay2}{
  morekeywords={
    QQ,ZZ,RR,ideal,dim,gens,degree,rank,
    matrix,det,trace,transpose,
    map,apply,select,
    substitute, jacobian
  },
  sensitive=true,
  morecomment=[l]{--},
  morestring=[b]"
}

\usepackage[utf8]{inputenc}
\usepackage{amsmath, amsthm, amssymb, color, mathbbol}
\numberwithin{equation}{section}
\usepackage{graphicx}
\usepackage{makecell}
\usepackage{multicol}
\usepackage{url}

\usepackage{parskip}
\usepackage[table,xcdraw]{xcolor}
\usepackage{tikz}
\usetikzlibrary{decorations.text}
\usepackage{array,amscd}
\usepackage{amsfonts}
\usepackage[T1]{fontenc}
\usepackage{appendix}
\usepackage[arrow, matrix, curve]{xy}
\usepackage{natbib}
\usepackage{booktabs}
\usepackage{siunitx}
\usepackage{multirow}
\usepackage[font=small,labelfont=bf]{caption}
\usepackage{caption}
\usepackage{subcaption}
\usepackage{blkarray}
\usepackage{algorithm}
\usepackage{algpseudocode}
\usepackage{tcolorbox}
\usepackage{hyperref}

\usepackage{tikz}
\usetikzlibrary{snakes}
\usetikzlibrary {shapes}
\usetikzlibrary {arrows}
\usetikzlibrary {positioning}
\usetikzlibrary {calc}

\usepackage{tikz-cd}
\usetikzlibrary{positioning, arrows.meta, fit}

\usepackage{pifont}

\usepackage[nameinlink, noabbrev]{cleveref}
\usepackage[shortlabels]{enumitem}
\usepackage{booktabs}
\usepackage{caption}
\usepackage{subcaption}

\newtheorem{theorem}{Theorem} 
\newtheorem{definition}{Definition}
\newtheorem{corollary}{Corollary}
\newtheorem{proposition}{Proposition}
\newtheorem{lemma}{Lemma}

\newtheorem{example}{Example}
\theoremstyle{definition}
\newenvironment{extprop}[1]
{\par\medskip
 \noindent\textbf{Proposition #1. }\normalfont}
{\par\medskip}

\newenvironment{extthm}[1]
{\par\medskip
 \noindent\textbf{Theorem #1. }\normalfont}
{\par\medskip}

\newcommand{\RR}{\mathbb{R}}

\newcommand{\CC}{\mathbb{C}}

\newcommand{\EE}{\mathbb{E}}
\newcommand{\balpha}{\boldsymbol{\alpha}}

\definecolor{darkgreen}{rgb}{0,0.4,0}
\definecolor{MyBlue}{rgb}{0,0.08,0.7} 
\definecolor{MyRed}{rgb}{0.85,0.08,0}

\DeclareMathOperator{\rk}{rank}

\DeclareMathOperator{\ci}{\perp\kern-1.3ex\perp}
\DeclareMathOperator{\nci}{\not\kern-0.3ex\ci}

\makeatletter
\def\blfootnote{\gdef\@thefnmark{}\@footnotetext}
\makeatother

\usepackage[activate={true,nocompatibility}, 
            final,              
            tracking=true,
            factor=1100,        
            stretch=10,         
            shrink=10]{microtype}
\SetTracking{encoding={*}, shape=sc}{40}

\title{Identifiability of partial-mastery cognitive diagnostic models}

\ifisarxive
    \renewcommand{\shorttitle}{Identifiability of PM-CDMs}
\fi

\author{Jun Wu$^{1}$, Patrícia Martinková$^{1,2}$, Elena Erosheva$^{3}$\\
\small $^{1}$ Institute of Computer Science of the Czech Academy of Sciences, Prague, Czech Republic\\
\small $^{2}$ Faculty of Education, Charles University, Prague, Czech Republic\\
\small $^{3}$ Department of Statistics, School of Social Work, and The Center for Statistics and the Social Sciences, \\University of Washington, Seattle, WA, USA\\
}

\date{July 4, 2026}

\begin{document}
\doublespacing

\maketitle

\begin{abstract}
Partial-mastery (PM) cognitive diagnostic models (CDMs) extend traditional CDMs by replacing binary latent attribute mastery indicators with continuous mastery scores for multiple latent attributes. In PM-CDMs, each subject is characterized by a fixed continuous latent mastery vector, from which item-specific binary attribute profiles are independently generated. This formulation provides a bridge between classical CDMs and continuous latent variable models.

Despite growing interest in PM-CDMs, their identifiability properties remain unexplored. In this work, we establish the first identifiability results for PM-CDMs. We derive sufficient conditions for identifiability that are direct analogues of established conditions for traditional CDMs. To develop the main argument, we use symbolic computation on a minimal example with five items and two latent attributes to show that the Jacobian of the model parameterization is generically nonzero. Combining tools from real analysis and algebraic statistics, we prove that this local property implies generic finite-to-one identifiability of the item parameters and the marginal distributions of the relevant latent attributes. We further show that if the $Q$-matrix contains such identifiable local structures for all attribute pairs, identifiability extends to the full PM-CDM. These findings provide a rigorous theoretical foundation for estimation and inference in partial-mastery cognitive diagnostic models.
\end{abstract}
\newpage


\section{Introduction}
Cognitive diagnostic models (CDMs) 
are widely used in educational and psychological assessment to provide fine-grained information about individuals' levels of specific attributes \citep{leighton2007cognitive, junker2001cognitive, haertel1989using}. A key feature of CDMs is their interpretability, as the latent attributes correspond to a set of predefined skills or traits. 
In this respect, CDMs align with a broader psychometric objective of moving beyond aggregate test scores toward interpretable representations of the latent constructs underlying observed responses   \citep{martinkova2023computational, rao2007psychometrics}.

Standard CDMs typically assume that each latent attribute is binary, indicating either mastery or non-mastery. Although this assumption facilitates model interpretation and estimation, it may be restrictive in practice. In many applications, mastery is better viewed as a continuum, and the distinction between mastery and non-mastery may be ambiguous. To address this limitation, partial-mastery CDMs (PM-CDMs) were proposed by \citet{shang2021partial}, allowing latent attributes to take continuous values between 0 and 1. This extension substantially increased modeling flexibility by accommodating intermediate levels of proficiency. However, it also introduced a more complex latent structure, expanding the latent space and raising new theoretical challenges, particularly with respect to identifiability. 

Identifiability is a fundamental property of statistical models that concerns whether model parameters can be uniquely recovered from the observed data distribution. Without identifiability, distinct parameter values may generate the same distribution of observed responses, rendering statistical estimation and inference ambiguous.

The identifiability of CDMs has been studied extensively. Existing results include conditions based on the number of items and latent attributes through Kruskal’s tensor decomposition theorem \citep{allman2009identifiability}, as well as structural conditions on the $Q$-matrix derived using linear algebraic techniques \citep{xu2017identifiability}. More recent work has further relaxed 
these structural requirements and established identifiability not only of the model parameters but also of the number of latent attributes and the $Q$-matrix itself \citep{lee2025identifiability}.

Despite these advances, existing identifiability results rely fundamentally on the assumption that the latent attributes are binary. Extending these results to PM-CDMs is nontrivial because the introduction of continuous latent mastery variables fundamentally alters the model structure. From another perspective, PM-CDM can be viewed as grade-of-membership models for item response data \citep{erosheva2002alternative,erosheva2005comparing,gu2023dimension} subject to additional structural constraints on the latent class probabilities. However, identifiability theory for such modes remains incomplete, and existing results do not directly yield identifiability of PM-CDMs. Consequently, the identifiability properties of PM-CDMs remain largely unexplored.

The present paper addresses this gap by developing the first identifiability theory for PM-CDMs. We derive sufficient conditions for identifiability that are direct analogues of well-established conditions for traditional CDMs, thereby extending classical identifiability theory to the partial-mastery setting. 
The paper proceeds as follows. In Section 2, we review the formulations of CDMs and PM-CDMs. In Section 3, we introduce the identifiability framework and review the notions of global, generic, and local identifiability. We then summarize key identifiability results for restricted latent class models (RLCMs) and show that PM-CDMs can be represented as highly overparametrized RLCMs, while demonstrating why existing RLCM identifiability results do not directly apply in this setting. Section 4 presents the main identifiability results and their proofs. In Section 5, we discuss applications of the theory, showing that our identifiability results extend to the recently proposed additive PM-CDMs and illustrating how they can be used to analyze identifiability in practical settings with well-known $Q$-matrix structures. Finally, Section 6 concludes with a discussion of the implications and limitations of the proposed identifiability results.


\section{Model Formulation}
\subsection{Cognitive Diagnostic Models}
Suppose that a test consists of \(J\) binary-response items measuring \(K\) latent attributes. For each subject, let
\[
\mathbf{y}=(y_1,\dots,y_J)
\]
denote the observed response vector, where \(y_j=1\) indicates a correct response to item \(j\). Each subject is associated with a latent attribute profile
\[
\boldsymbol{\alpha}=(\alpha_1,\dots,\alpha_K)\in\{0,1\}^K,
\]
where \(\alpha_k=1\) indicates mastery of the \(k\)-th attribute and \(\alpha_k=0\) otherwise.

The population distribution of attribute profiles is characterized by the probability vector
\[
\mathbf{p}_{\boldsymbol{\alpha}}
=
\bigl(
p_{\boldsymbol{\alpha}}
:
\boldsymbol{\alpha}\in\{0,1\}^K
\bigr)^T,
\]
satisfying
\[
\sum_{\boldsymbol{\alpha}\in\{0,1\}^K}
p_{\boldsymbol{\alpha}}=1.
\]

Given an attribute profile $\boldsymbol{\alpha}$, the response $y_j$ follows a Bernoulli distribution with success probability \(\theta_{j,\boldsymbol{\alpha}}=P(y_j=1\mid\boldsymbol{\alpha})\).
The response probabilities $\theta_{j,\boldsymbol{\alpha}}$ are constrained by the relationship between items and attributes, encoded by the $Q$-matrix \(Q=(q_{jk})_{J\times K}\), where $q_{jk}=1$ indicates that item $j$ requires attribute $k$. Let $\mathbf q_j$ denote the $j$th row of $Q$, and let \(\mathbf e_k=(0,0\dots,1,\dots,0)\) denote the $k$th standard basis row vector in $\RR^K$. An item $j$ is called a \textit{pure item} if it requires only one latent attribute, that is, if $\mathbf q_j=\mathbf e_k$ for some $k$. Pure items play a central role in the identifiability results developed later.

Common constraints include monotonicity of the response probabilities,
$\theta_{j,\boldsymbol{\alpha}}\ge\theta_{j,\boldsymbol{\alpha}'}$ for $\boldsymbol{\alpha}\succeq\boldsymbol{\alpha}'$, where $\succeq$ denotes element-wise ordering. Another common constraint is invariance with respect to non-required attributes: $\theta_{j,\boldsymbol{\alpha}}=\theta_{j,\boldsymbol{\alpha}'}$
if $
\boldsymbol{\alpha}\odot\mathbf q_j
=
\boldsymbol{\alpha}'\odot\mathbf q_j,
$
where $\odot$ denotes element-wise multiplication. The following strict monotonicity assumption is also important for identifiability: for any latent dimension $k$ and item $j$ with $\mathbf q_j=\mathbf{e}_k$, $\theta_{j,\mathbf e_k}>\theta_{j,\mathbf 0}.$ Throughout the paper, the item parameters are assumed to satisfy these constraints.

Under the conditional independence assumption,
\begin{align*}
{P}(\mathbf{y} \mid \boldsymbol{\alpha})
= \prod_{j=1}^J \theta_{j,\boldsymbol{\alpha}}^{y_j}
(1 - \theta_{j,\boldsymbol{\alpha}})^{1-y_j}.
\end{align*}
The marginal probability of the response vector $\mathbf y$ is obtained by summing over all latent attribute profiles $\balpha$:
\begin{align*}
{P}(\mathbf{y})
= \sum_{\boldsymbol{\alpha} \in \{0,1\}^K}
p_{\boldsymbol{\alpha}}
\prod_{j=1}^J \theta_{j,\boldsymbol{\alpha}}^{y_j}
(1 - \theta_{j,\boldsymbol{\alpha}})^{1-y_j}.
\end{align*}

Among the many cognitive diagnosis models proposed in the literature, the Deterministic Input Noisy output "And" gate (DINA) model \citep{junker2001cognitive} and the Generalized DINA (GDINA) model \citep{de2011generalized} are two of the most widely used. 
For each item, the reduced attribute profile consists of the attributes required by that item. In the DINA model, all reduced attribute profiles that do not master every required attribute share the same response probability, while a different response probability is assigned to the fully mastered profile.
In contrast, the GDINA model assigns a distinct response probability to each reduced attribute profile.

\subsection{Partial-Mastery Cognitive Diagnostic Models}
In standard CDMs, attributes are typically represented as binary mastery indicators, where each attribute is either mastered or not. To allow for intermediate levels of attribute mastery, \citet{shang2021partial} proposed partial-mastery cognitive diagnostic models (PM-CDMs), which generalize binary attribute mastery indicators to continuous mastery levels. In these models, each subject has $K$ latent mastery scores $d_k\in[0,1]$. These scores represent continuous levels of attribute mastery, ranging from not mastered ($d_k=0$) to fully mastered ($d_k=1$), allowing for greater flexibility in modeling intermediate proficiency levels. The latent mastery score vector $\mathbf{d}$ is assumed to follow a Gaussian copula model to capture dependencies among the latent attributes:
\begin{equation}
\left\{z_k:=\Phi^{-1}(d_k);\ k=1,\dots,K\right\}^T\sim N(\boldsymbol{\mu},\boldsymbol{\Sigma}),
\end{equation}
where $\Phi^{-1}$ is the inverse cumulative distribution function of a standard normal distribution. Consequently, each marginal mastery score $d_k$ is uniformly distributed on [0,1], while dependencies among attributes are governed by \((\boldsymbol{\mu},\Sigma)\).

In addition, PM-CDMs allow latent attributes to be probabilistically manifested at the item level according to the mastery score vector \(\mathbf d\).

For a fixed mastery score vector $\mathbf d$, a realized attribute profile
$\boldsymbol{\alpha}\mid\mathbf d$ is generated for each item by independently drawing
\(\alpha_k\sim\mathrm{Bernoulli}(d_k), \ k=1,\dots,K\). 
Therefore,
\[
p_{\boldsymbol{\alpha}\mid\mathbf d}
=
\prod_{k=1}^K
d_k^{\alpha_k}(1-d_k)^{1-\alpha_k}.
\]

The marginal response probability to item $j$ is constructed by averaging
$\theta_{j,\boldsymbol{\alpha}}$ over the conditional distribution of the realized attribute profile given $\mathbf d$:
\begin{equation}\label{theta_jd}
\theta_{j,\mathbf{d}}=\sum_{{\boldsymbol{\alpha}}\in\{0,1\}^K}\theta_{j,{\boldsymbol{\alpha}}}\cdot p_{{\boldsymbol{\alpha}}|\mathbf{d}}.
\end{equation}
Thus, $\theta_{j,\mathbf{d}}$ is a polynomial function of the mastery scores $\mathbf d$.
Here \(\theta_{j,\boldsymbol{\alpha}} = P(Y_j=1\mid\boldsymbol{\alpha})\)
denotes the response probability associated with the latent class
$\boldsymbol{\alpha}\in\{0,1\}^K$ in the corresponding CDM. The PM-DINA and PM-GDINA models inherit the corresponding response probability structures from the DINA and GDINA models, respectively.

The joint response probability under the PM-CDM is then obtained by integrating over the distribution of \(\mathbf d\):
\begin{equation}\label{pmcdm_lh}
P_{\text{PM-CDM}}(\mathbf{y})=\int_{\mathbf{d}\in[0,1]^K}\prod_{j=1}^{J}\theta_{j,\mathbf{d}}^{y_j}(1-\theta_{j,\mathbf{d}})^{1-y_j}dD_{\mu,\Sigma}(\mathbf{d}).
\end{equation}
The model parameters, therefore, consist of the CDM item parameters $\{\theta_{j,\boldsymbol{\alpha}}\}$ and the copula parameters $(\mu,\Sigma).$

Compared with standard CDMs, PM-CDMs introduce a substantially richer latent structure through the continuous mastery variables and the copula parameters $(\boldsymbol{\mu},\boldsymbol{\Sigma})$. This additional flexibility raises fundamental questions about whether the model parameters can be uniquely recovered from the observed response distribution. The next section formalizes this identifiability problem.


\section{Preliminaries}
\subsection{Notions of Identifiability}
This section reviews the notions of identifiability and the Jacobian criterion for local identifiability.

\begin{definition}(Definition 16.1.1 in \cite{sullivant2018algebraic})
Let the parameter space $\Omega\subseteq\RR^m$ be a semi-algebraic set \footnote{A semi-algebraic set is a subset of Euclidean space defined by finitely many polynomial equalities and inequalities.}, and let $\phi:\Omega\to M$ be a rational parameterization map defined everywhere on $\Omega$, where each component of $\phi$ can be written as a ratio of polynomial functions of the model parameters. Let 
$M=\operatorname{im}(\phi)$.
The parameterization map $\phi$ is said to be
\begin{itemize}
    \item \textbf{globally identifiable} if $\phi$ is one-to-one on $\Omega$;
    \item \textbf{generically identifiable} if
    \(\phi^{-1}(\phi(\theta))=\{\theta\}\)
    for almost all \(\theta\in\Omega\);
    \item \textbf{locally identifiable} if
    \(|\phi^{-1}(\phi(\theta))|<\infty\)
    for almost all $\theta\in\Omega$.
\end{itemize}
\end{definition}

In latent variable models, the parameter vector typically consists of item parameters and parameters governing the distribution of latent variables, denoted by
\[
\theta=(\Theta,\mathbf p).
\]
For models with $J$ binary responses, the parameterization map induces a probability distribution on the simplex $\Delta_{2^J}$,
corresponding to the $2^J$ possible response patterns.
\begin{align}\label{eq_map}
\phi(\Theta,\mathbf p)
=
\bigl(
P(\mathbf Y=\mathbf y \mid \Theta,\mathbf p)
:
\mathbf y\in\{0,1\}^J
\bigr).
\end{align}

In CDMs, the parameterization map is additionally determined by the $Q$-matrix. The identifiability of $Q$ may be studied jointly with that of the item parameters $\Theta$, or under the assumption that $Q$ is known. In this work, we assume that $Q$ is known. Since the form of the parameterization map is determined by the zero pattern of $Q$, we write $\phi$ instead of $\phi_Q$ when no ambiguity arises.

Global identifiability is the strongest notion of identifiability and is often referred to as strict identifiability in the latent variable model literature. Establishing global identifiability is typically challenging, and the proof techniques often depend heavily on the specific model structure. 
For example, identifiability results for factor analysis models rely on matrix-analytic arguments \citep{anderson1956statistical}, whereas proofs for restricted latent class models (RLCMs) involve carefully constructed algebraic manipulations and cancellation arguments \citep{xu2017identifiability}.

Generic identifiability relaxes global identifiability by allowing the exclusion of singular parameter values lying in a measure-zero set. Under this notion, polynomial expressions arising in the parameterization map can be assumed to be nonzero for generic parameter values \citep{okamoto1973distinctness}. 
For example, \citet{gu2024blessing} shows that a class of CDMs is identifiable outside a measure-zero subset corresponding to conditional independence among certain latent attributes.

Local identifiability is the weakest notion among the three. Under this definition, the parameterization map is only required to be locally one-to-one, and the preimage \(\phi^{-1}(\phi(\theta))\) may contain finitely many distinct parameter points in $\mathbb{R}^m$. This notion of local identifiability, commonly used in algebraic statistics, is stronger than the definition appearing in parts of the factor analysis literature, such as \cite{shapiro1985identifiability}. The latter only requires the parameterization map to be locally one-to-one and does not impose finiteness of the preimage.

A standard approach to studying local identifiability is to examine the generic rank of the Jacobian matrix $J(\phi)$.

\begin{extprop}{16.1.7 in \cite{sullivant2018algebraic}}
Let $\Omega\subseteq\RR^m$ and let $\phi$ be a rational map. Then the dimension of the model $\dim(\operatorname{im}\phi)$ is equal to the rank of the Jacobian matrix evaluated at a generic point:
\begin{align*}
J(\phi) = \begin{pmatrix}
\frac{\partial \phi_1}{\partial \theta_1} & \cdots & \frac{\partial \phi_1}{\partial \theta_m}\\
\vdots & \ddots & \vdots\\
\frac{\partial \phi_r}{\partial \theta_1} & \cdots & \frac{\partial \phi_r}{\partial \theta_m}\\
\end{pmatrix}, \quad \dim(\operatorname{im}\phi)
=
\rk J(\phi).
\end{align*}
The parameter vector $\theta$ is locally identifiable if \(\rk(J(\phi))=m\).
\end{extprop}

For example, \cite{perez2023identifiability} analyzed the local identifiability of the simplest globally identifiable cognitive diagnostic model, consisting of \(K=1\) attribute and \(J=3\) items.

\subsection{Marginal $T$-Matrix and Identifiability Results for RLCMs} 
This subsection reviews the marginal $T$-matrix formulation and several identifiability results for RLCMs from \cite{xu2017identifiability}, which we will use in the subsequent analysis.

In a restricted latent class model (RLCM) with $J$ items and $K$ binary latent attributes, the $T$-matrix \(T(Q,\Theta)\) is a $2^J\times 2^K$ matrix indexed by the response vector \(\mathbf r\in\{0,1\}^J\) and the latent attribute profile \(\boldsymbol{\alpha}\in\{0,1\}^K\). 
The \((\mathbf{r}, \boldsymbol{\alpha})\) entry of \(T(Q, \Theta)\), denoted by \(t_{\mathbf{r},\boldsymbol{\alpha}}(Q, \Theta)\), is the marginal probability of the event $\mathbf{y}\succeq\mathbf{r}$ given the latent attribute profile \(\boldsymbol{\alpha}\) under the model specified by \((Q,\Theta)\):
\[
t_{\mathbf{r},\boldsymbol{\alpha}}(Q, \Theta)=P(\mathbf{y}\succeq \mathbf{r}\mid Q,\Theta, \boldsymbol{\alpha}).
\]
For $\mathbf{r}=\mathbf{e}_j$, 
\begin{align*}
t_{\mathbf{e}_j, \boldsymbol{\alpha}}(Q, \Theta)=P(y_j=1\mid Q, \Theta, \boldsymbol{\alpha})=\theta_{j,\boldsymbol{\alpha}}.    
\end{align*}
Let \(T_{\mathbf{r},\cdot}\) denote the row vector \(T(Q,\Theta)\) corresponding to $\mathbf r$. Under the conditional independence assumption given $\boldsymbol{\alpha}$,
\begin{align}\label{T_decomp}
T_{\mathbf{r},\cdot}(Q,\Theta) = \bigodot_{j:o_j=1}T_{\mathbf{e}_j,\cdot}(Q, \Theta).  
\end{align}

Multiplying the $T$-matrix by the latent class probability vector $\mathbf{p}$ yields the marginal probabilities of positive responses to subsets of items:
\begin{align}\label{eq_marginalprob}
\forall \mathbf{r}\in\{0,1\}^{J}, \quad T_{\mathbf{r},\cdot}(Q,\Theta)\mathbf{p} = \sum_{\boldsymbol{\alpha}}t_{\mathbf{r},\boldsymbol{\alpha}}(Q,\Theta)\mathbf{p} = P(\mathbf{y}\succeq \mathbf{r}\mid Q, \Theta, \mathbf{p}).   
\end{align}
There is a one-to-one correspondence between the probabilities \(P(\mathbf{y}\succeq \mathbf{r}\mid Q, \Theta, \mathbf{p})\) and \(P(\mathbf{y}= \mathbf{r}\mid Q, \Theta, \mathbf{p})\) over
\(\mathbf r\in\{0,1\}^J\). Therefore, identifiability can be studied through the equation
\begin{align*}
T(Q,\Theta)\mathbf{p}=T(Q,\overline\Theta)\overline{\mathbf{p}}.
\end{align*}
In particular, to establish global identifiability, it suffices to show that
\begin{align*}
T(Q,\Theta)\mathbf{p}=T(Q,\overline\Theta)\overline{\mathbf{p}}\implies (\Theta, \mathbf{p})=(\overline{\Theta}, \overline{\mathbf{p}}).
\end{align*}

The following proposition serves as a useful technical tool in the proofs of identifiability results:
\begin{extprop}{1 in \cite{xu2017identifiability}}
\((\Theta, \mathbf{p})\) is identifiable if and only if for any \((\Theta, \mathbf{p})\neq (\Theta', \mathbf{p}')\), there exists \(\mathbf{r}\in\{0,1\}^J\) such that 
\begin{align}\label{eq_id1}
T_{\mathbf{r},\boldsymbol{\cdot}}(Q,\Theta) \mathbf{p}\neq T_{\mathbf{r},\boldsymbol{\cdot}}(Q,\overline{\Theta})\overline{\mathbf{p}}.    
\end{align}
\end{extprop}
By analyzing the $T$-matrix, \cite{xu2017identifiability} provides a sufficient condition for the identifiability of \((\Theta, \mathbf{p})\) in RLCMs.
\begin{extthm}{1 in \cite{xu2017identifiability}}
In an RLCM, if the $Q$-matrix takes the form
\[
Q=
\begin{pmatrix}
I_K \\
I_K \\
Q'
\end{pmatrix},
\]
and if, for any \(k\in \{1, \cdots, K\}\), \((\theta_{j,\mathbf{e}_k}; j>2K)^T\neq (\theta_{j,\mathbf{0}}; j>2K)^T\), then \((\Theta, \mathbf{p})\) is globally identifiable.
\end{extthm}
The condition on the items in $Q'$ requires that, for each attribute $k$, there exists at least one item in $Q'$ for which the latent classes $\boldsymbol{\alpha}=\mathbf e_k$ and $\boldsymbol{\alpha}=\mathbf 0$ yield different positive response probabilities; equivalently, the two vectors above differ in at least one coordinate. This condition holds generically whenever $q_{jk}\neq 0$ for some $j>2K$.

The requirement that the $Q$-matrix contains two identity submatrices is among the most influential sufficient conditions in the CDM identifiability literature and serves as a benchmark for the results developed later.

The following proposition plays a key role in the proof of Theorem 1 in \cite{xu2017identifiability} and will also be useful in our subsequent analysis.
\begin{proposition}{(Proposition 3 in \cite{xu2017identifiability})}\label{prop3_xu}
For any \(\boldsymbol{\theta}^*=(\theta_1^*, \cdots, \theta_J^*)^T\in\RR^J\), there exists an invertible matrix \(D(\boldsymbol{\theta}^*)\) depending only on $\boldsymbol{\theta}^*$, such that the matrix \(D(\boldsymbol{\theta}^*)\) is lower triangular with unit diagonal entries, 
and
\[
T(Q,\Theta-\boldsymbol{\theta}^*\mathbf{1}^T) = D(\boldsymbol{\theta}^*)T(Q, \Theta).
\]
\end{proposition}

\subsection{Identifiability Problem for PM-CDMs}
Although PM-CDMs can be represented as restricted latent class models, the resulting representation differs fundamentally from the setting considered by \cite{xu2017identifiability}, and existing identifiability results cannot be applied directly.

The joint response distribution of a PM-CDM admits an equivalent RLCM representation \citep{shang2021partial,erosheva2007describing}. 
Consider the stacked latent binary vector
\[
A=(\boldsymbol{\alpha}_1^*,\dots,\boldsymbol{\alpha}_J^*),
\qquad
\boldsymbol{\alpha}_j^*\in\{0,1\}^K.
\]
The latent space of $A$ is
\[
\mathcal A
=
\prod_{j=1}^J\{0,1\}^K
=
\{0,1\}^{JK},
\]
which has cardinality $2^{JK}$.
The latent class probability associated with $A$ is
\begin{align}\label{pi_A}
\pi_A
&=
\mathbb E_{D_{\mu,\Sigma}}
\left[
\prod_{j=1}^J
\prod_{k=1}^K
d_k^{\alpha_{jk}^*}
(1-d_k)^{1-\alpha_{jk}^*}
\right]
\nonumber\\
&=
\mathbb E_{D_{\mu,\Sigma}}
\left[
\prod_{k=1}^K
d_k^{\sum_{j=1}^J\alpha_{jk}^*}
(1-d_k)^{\sum_{j=1}^J\left(1-\alpha_{jk}^*\right)}
\right].
\end{align}

Conditional on the latent class $A$, the response distribution for item $j$ depends only on the item-specific latent attribute profile $\boldsymbol{\alpha}_j^*$:
\[
P(y_j=1\mid A)={P}(y_j=1\mid \boldsymbol{\alpha}_j^*)=\theta_{j,\boldsymbol{\alpha}_j^*}.
\]
That is, the conditional response probability coincides with the corresponding response probability in the original CDM. 
The resulting RLCM representation has probability mass function
\[
{P}_{\textrm{RLCM}}(\mathbf{y}\mid \Theta, \boldsymbol{\mu}, \boldsymbol{\Sigma})=\sum_{A=(\boldsymbol{\alpha}_1^*,\dots,\boldsymbol{\alpha}_J^*)\in\mathcal{A}}\pi_A\prod^J_{j=1}\theta_{j,\boldsymbol{\alpha}_j^*}^{y_j}(1-\theta_{j,\boldsymbol{\alpha}_j^*})^{1-y_j}.   
\]
This representation can therefore be viewed as an RLCM with $2^{JK}$ latent classes and $J$ items.

The identifiability theorem, therefore, implies that if the $Q$-matrix takes the form 
\begin{align*}
Q=
\begin{pmatrix}
I_{JK} \\
I_{JK} \\
Q'
\end{pmatrix},
\end{align*}
and each column of $Q'$ contains at least one nonzero entry, then all latent class probabilities \(\{\pi_A:A\in\mathcal A\}\) and item response probabilities
\(\{\theta_{j,\boldsymbol{\alpha}}\}\) are globally identifiable. 
However, this structural condition on the $Q$-matrix cannot hold for the PM-CDM, since the induced RLCM has $JK$ latent attributes but only $J$ items. Therefore, a different approach is needed to study the identifiability of PM-CDMs.

The response distribution of $J$ binary items lies in a simplex of dimension $2^J-1$. However, the induced RLCM representation introduces $2^{JK}$ latent class probability parameters, far exceeding the dimension of the observable model. Consequently, the RLCM representation appears severely overparameterized from the viewpoint of identifiability analysis.

Nevertheless, these probabilities are not free parameters in the PM-CDM. By \eqref{pi_A}, they are fully determined by the Gaussian copula parameters \((\boldsymbol{\mu},{\Sigma})\) and therefore lie in a lower-dimensional subset of the $2^{JK}$-class probability simplex.

For \(A=(\boldsymbol{\alpha}_1^*,\dots,\boldsymbol{\alpha}_J^*)\), the integral representation is
\begin{align}\label{p_formula}
\pi_A
=
\int_{\mathbb R^K}
&
\prod_{k=1}^K
\Phi(z_k)^{\sum_{j=1}^J\alpha_{jk}^*}
\bigl(1-\Phi(z_k)\bigr)^{\sum_{j=1}^J(1-\alpha_{jk}^*)}
\nonumber\\
&\cdot
(2\pi)^{-K/2}
|\boldsymbol{\Sigma}|^{-1/2}
\exp\left(
-\frac12
(\mathbf z-\boldsymbol\mu)^T
\boldsymbol\Sigma^{-1}
(\mathbf z-\boldsymbol\mu)
\right)
\, d\mathbf z.
\end{align}
The latent class probabilities $\pi_A$ are determined by Gaussian copula integrals and therefore are not rational functions of \((\boldsymbol{\mu},\boldsymbol{\Sigma})\). Consequently, the parameter space associated with \((\Theta,\boldsymbol{\mu},\boldsymbol{\Sigma})\) is not semi-algebraic. This distinguishes PM-CDMs from the latent class models typically studied in algebraic statistics.
As a result, the algebraic tools for identifiability developed in \cite{sullivant2018algebraic} cannot be directly applied to the parameterization in terms of $(\Theta,\boldsymbol{\mu},\boldsymbol{\Sigma})$. To overcome this obstacle, the next section develops an alternative parameterization that enables the use of algebraic identifiability techniques.


\section{Main Identifiability Results}\label{sec_theorem}
\subsection{Main Theorem}
\begin{theorem}\label{main_theorem}
The parameters of a PM-CDM,
\[
\{\theta_{j,\boldsymbol{\alpha}}\}
\quad\text{and}\quad
(\boldsymbol{\mu},{\Sigma}),
\]
are locally identifiable if one of the following conditions holds:
\begin{enumerate}[(i)]
\item $K=1$ and $J\ge 4$,
\item $K\ge 2$ and the $Q$-matrix takes the form
\[
\begin{pmatrix}
I_K\\
I_K\\
I_K\\
Q'
\end{pmatrix},
\]
where $Q'$ is an arbitrary matrix of $K$ columns. That is, each latent attribute is measured by at least three pure items.
\end{enumerate}
\end{theorem}
The proof proceeds by first identifying a core set of item and latent-distribution parameters within a reduced submodel. These identified quantities then serve as anchors for recovering the remaining parameters in subsequent steps.

The key observation is that, although the induced RLCM representation contains $2^{JK}$ latent classes, the corresponding class probabilities exhibit substantial structural redundancy -- many latent classes have identical class probabilities. 
The first step of the proof exploits this redundancy within a collection of suitably chosen submodels. These submodels yield more algebraically tractable parameterizations and identify the item parameters for the first $3K$ items. 
In the second step, the copula parameters \((\boldsymbol{\mu},\Sigma)\) are identified. The final step establishes identifiability of the remaining item parameters associated with $Q'$.

\begin{figure}[h]
    \centering
    \includegraphics[width=0.6\textwidth]{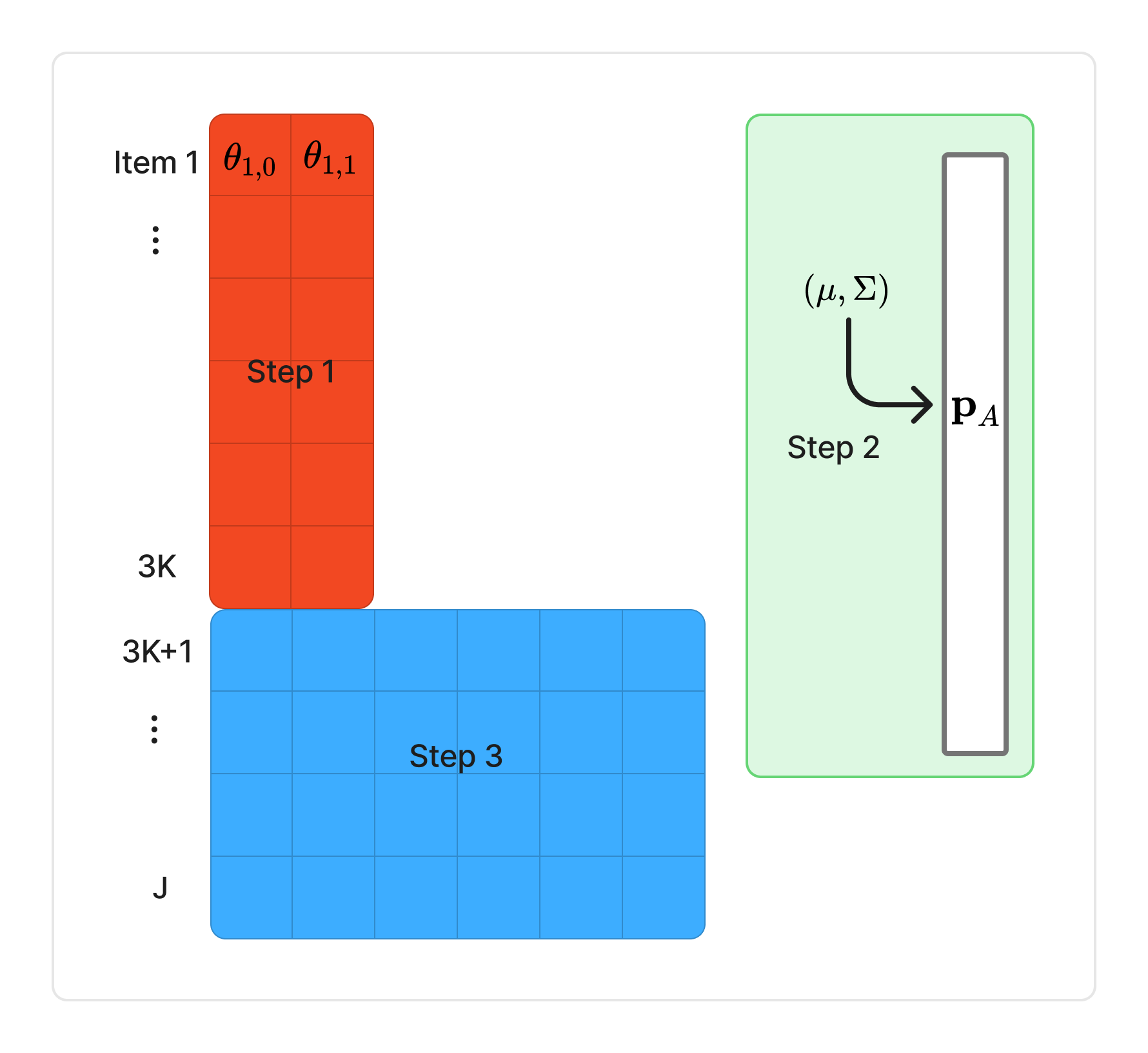}
    \caption{Overview of the proof strategy for Theorem 1. Step 1 establishes local identifiability of a reduced submodel, Step 2 identifies the copula parameters \((\boldsymbol{\mu}, \Sigma)\), and Step 3 recovers the remaining item parameters.}
    \label{fig:proof-sketch}
\end{figure}

\subsection{Step 1: A Submodel with Full-rank Jacobian}
In this step, we establish local identifiability of a specific submodel via the Jacobian criterion.

\subsubsection{Reparameterization of $\mathbf{p}$ in a Submodel}

For items depending on only one latent attribute, the latent class probabilities exhibit additional symmetry that substantially reduces the number of free parameters. Exploiting this symmetry is essential because it transforms the highly overparameterized latent-class representation into a parameterization of manageable dimension.

Consider the realized binary sequence associated with the latent representation profiles across items. For item $j$, the $j$th entry equals 1 if the required attribute is realized from the Bernoulli draw of the mastery score, and equals 0 otherwise. The items can be partitioned into at most $L=2^K-1$ blocks according to the attribute pattern they require. Within each block, latent-class probabilities depend only on the number of realized attributes and not on their specific arrangement.

Suppose the $l$th block contains $c_l$ items. Since items within the same block depend on the same latent attributes, the probability of observing $o_l$ ones and $c_l-o_l$ zeros in this block depends only on the counts and not on the specific positions of the ones and zeros. Consequently, the latent class probabilities can be reparameterized by the unique values
\[
p^u_{t_1\cdots t_L},
\qquad
t_l\in\{0,1,\dots,c_l\}.
\]

\begin{example}\label{eg_5items}
We study the case $K=2$ and $J=5$, which is the smallest nontrivial setting that captures the key combinatorial structure of the general argument. 
Consider the $Q$-matrix
\begin{equation}\label{Q1}
Q_1=
\begin{pmatrix}
1 & 0\\
1 & 0\\
1 & 0\\
0 & 1\\
0 & 1
\end{pmatrix}.
\end{equation}
The items are divided into two blocks according to the required attribute pattern. The first three items form the first block and the last two items form the second block, so that $c_1=3$ and $c_2=2$.

Within each block, latent representation sequences with the same number of ones have identical class probabilities. The equivalence classes for the two blocks are
\[
\Bigl\{
0:\{000\},\
1:\{100,010,001\},\
2:\{110,101,011\},\
3:\{111\}
\Bigr\},
\]
and
\[
\Bigl\{
0:\{00\},\
1:\{10,01\},\
2:\{11\}
\Bigr\},
\]
respectively, yielding $4\times 3=12$ distinct latent class probabilities.

The reduced latent class probabilities satisfy
\[
p_{o_1o_2o_3o_4o_5}
=
p_{o_1'o_2'o_3'o_4'o_5'}
=
p^u_{(o_1+o_2+o_3),(o_4+o_5)},
\]
whenever
\[
o_1o_2o_3 \sim o_1'o_2'o_3',
\qquad
o_4o_5 \sim o_4'o_5',
\]
that is, whenever the two sequences contain the same number of ones within each block.

We denote these distinct latent class probabilities by
$
p^u_{00}, p^u_{10}, p^u_{20}, p^u_{30},
p^u_{01}, p^u_{11}, p^u_{21}, p^u_{31},
p^u_{02}, p^u_{12}, p^u_{22}, p^u_{32}.
$
The first subscript denotes the number of ones among the first three items, and the second subscript denotes the number of ones among the last two items. 

Define
\[
s(\mathbf p^u)
:=
p^u_{00}
+3p^u_{10}
+3p^u_{20}
+p^u_{30}
+p^u_{01}
+3p^u_{11}
+3p^u_{21}
+p^u_{31}
+p^u_{02}
+3p^u_{12}
+3p^u_{22}.
\]
The remaining probability is determined by the normalization constraint,
\[
p^u_{32}=1-s(\mathbf p^u).
\]
Hence the domain of the reduced latent class probabilities is
\[
\Delta
:=
\left\{
\mathbf p^u:
\mathbf p^u\succeq\mathbf0,
\ s(\mathbf p^u)\le1
\right\}.
\]
The corresponding reduced latent class probabilities and correct response probabilities can be arranged into a $5\times 32$ matrix, whose rows are indexed by the items and columns are indexed by all possible realized binary sequences. 
The matrix is used to construct the $31$ coordinates of the rational map, defined by~\eqref{eq_map} and~\eqref{eq_marginalprob}:
\[
T_{\mathbf r,\cdot}(Q,\Theta)\mathbf p = P(\mathbf y\succeq \mathbf r),
\qquad
\mathbf r\in\{0,1\}^5\setminus\{00000\},
\]
We display the 32 realized binary attribute profiles and the reduced latent class probabilities here. The explicit $5\times 32$ matrix \(T_{I_5,\cdot}(Q,\Theta)\) is provided in Appendix~\ref{TI5}. 

\begin{align*}
\begin{blockarray}{cccccccc}
00000 & 10000 & 01000 & 11000 & 00100 & 10100 & 01100 & 11100\\
p^u_{00} & p^u_{10} & p^u_{10} & p^u_{20} & p^u_{10} & p^u_{20} & p^u_{20} & p^u_{30}
\end{blockarray}
\end{align*}
\begin{align*}
\begin{blockarray}{cccccccc}
00010 & 10010 & 01010 & 11010 & 00110 & 10110 & 01110 & 11110\\
p^u_{01} & p^u_{11} & p^u_{11} & p^u_{21} & p^u_{11} & p^u_{21} & p^u_{21} & p^u_{31}
\end{blockarray}
\end{align*}
\begin{align*}
\begin{blockarray}{cccccccc}
00001 & 10001 & 01001 & 11001 & 00101 & 10101 & 01101 & 11101\\
p^u_{01} & p^u_{11} & p^u_{11} & p^u_{21} & p^u_{11} & p^u_{21} & p^u_{21} & p^u_{31}
\end{blockarray}
\end{align*}
\begin{align*}
\begin{blockarray}{cccccccc}
00011 & 10011 & 01011 & 11011 & 00111 & 10111 & 01111 & 11111\\
p^u_{02} & p^u_{12} & p^u_{12} & p^u_{22} & p^u_{12} & p^u_{22} & p^u_{22} & p^u_{32}
\end{blockarray}
\end{align*}
%
The rational map can be computed by the $T$-matrix decomposition~\eqref{T_decomp}, and was implemented in \texttt{Macaulay2} (see the Online Supplementary Material).
Evaluating the Jacobian matrix $J(\phi)$ at a specific parameter point $(\Theta_0, \mathbf{p}_0)$ shows that its rank is 21, equal to the number of free parameters in the submodel. 
\end{example}

The computation suggests that the Jacobian matrix \(J(\phi)\) has full column rank at generic parameter points. Consequently, for generic \((\Theta,\mathbf p)\), there exist at most finitely many \((\Theta',\mathbf p')\in\CC^{21}\)\footnote{\(\CC^n\) denotes the \(n\)-dimensional complex vector space.} such that
\(
T_{\mathbf r,\cdot}(Q,\Theta)\mathbf p
=
T_{\mathbf r,\cdot}(Q,\Theta')\mathbf p'
\)
for all \(\mathbf r\in\{0,1\}^5\).
The same conclusion holds when $\CC^{21}$ is replaced by $\RR^{21}$ or by \((0,1)^{10}\times\Delta\), where $\Delta$ denotes the domain of the reduced latent class probabilities $\mathbf p^u$:
\[
\Delta:=\{\mathbf{p}^u=(p^u_{00},p^u_{10},p^u_{20},p^u_{30},p^u_{01},p^u_{11},p^u_{21},p^u_{31},p^u_{02},p^u_{12},p^u_{22}):\mathbf{p}^u\succeq \mathbf{0},\ s(\mathbf p^u) \leq 1 \}.
\]
Denote \(\Omega=(0,1)^{10}\times\Delta\),
and let \(V=\left\{
(\Theta,\mathbf p^u)\in\Omega:
\rk(J(\phi))<21
\right\}\).
Example~\ref{eg_5items} shows that $V\neq\Omega$. Since the entries of $J(\phi)$ are rational functions of \((\Theta,\mathbf p^u)\), the set $V$ has Lebesgue measure zero in $\Omega$ (see the lemma in \citet{okamoto1973distinctness}).

Let $\Delta'\subseteq\Delta$ denote the subset parameterized by the copula parameters $(\boldsymbol{\mu},\boldsymbol{\Sigma})$, and let $\Omega'=(0,1)^{10}\times\Delta'$.
Although the rank-deficient set \(V
=
\left\{
(\Theta,\mathbf p^u)\in\Omega:
\rk(J(\phi))<21
\right\}
\)
has measure zero in $\Omega$, this does not immediately imply that $V\cap\Omega'$ has measure zero in $\Omega'$, since $\Omega'$ itself is a lower-dimensional subset of $\Omega$. Therefore, additional arguments are required to establish local identifiability for the submode.

\subsubsection{Real Analytic Functions}
We use properties of real analytic functions to prove that $V\cap \Omega'$ has measure zero in $\Omega'$.

\begin{definition}{(Definition 2.2.1 in \cite{krantz2002primer}}
Let $U\subseteq\RR^m$ be an open set. A function
$f:U\to\RR$
is called real analytic on $U$, written \(f\in C^\omega(U)\), if for every $x\in U$, the function $f$ admits a convergent power series expansion in a neighborhood of $x$.
\end{definition}

We will apply the following multivariate analogue of Corollary 1.2.6 in \cite{krantz2002primer}. The result can be viewed as a real-analytic-function version of the Okamoto lemma. It also appears as Exercise 3.1.1 in \cite{curry2025tasty} and follows from multivariate power series expansions together with coefficient comparison.

\begin{corollary}
Let $U\subseteq\RR^m$ be connected and open, and let
\(f\in C^\omega(U)\).
If $f=0$ on a nonempty open subset of $U$, then $f\equiv0$ on $U$.
\end{corollary}
For a real analytic function $f$, if its zero set $\{x:f(x)=0\}$ contains a nonempty open set, then $f\equiv 0$. As a consequence, if $f$ is nonzero at some point, then its zero set has measure zero. 
In Example \ref{eg_5items}, the point \((\Theta_0, \mathbf{p}^u_0)\) such that \(\rk(J)=21\) is obtained from 
\[
\Theta_0=\left(\frac{1}{5}, \frac{4}{5}, \frac{1}{4}, \frac{4}{5}, \frac{3}{20}, \frac{3}{4}, \frac{3}{10}, \frac{9}{10}, \frac{1}{5}, \frac{3}{5}\right),\qquad \boldsymbol{\mu}_0=
\begin{pmatrix}
0\\
0
\end{pmatrix},
\qquad {\Sigma}_0=
\begin{pmatrix}
 1 & 0\\
 0 & 1
\end{pmatrix}.
\]
Hence, there exists a $21\times21$ minor of $J(\phi)$ that is nonzero at this point. If this minor is real analytic as a function of $(\Theta, \boldsymbol{\mu}, \Sigma)$, then it is nonzero almost everywhere on the subspace $\Omega'$.
Consequently, the Jacobian matrix has full column rank almost everywhere on this subspace, implying local identifiability of $(\Theta,\mathbf p^u)$.

The sum, product, quotient, and composition of real analytic functions are again real analytic (Propositions 2.2.2 and 2.2.8 in \cite{krantz2002primer}). We already know that the selected $21\times21$ minor is real analytic on $(0,1)^{10}\times\Delta$. Therefore, to show that this minor is analytic as a function of \((\Theta,\boldsymbol{\mu},\boldsymbol{\Sigma})\), it suffices to prove that the map
\((\boldsymbol{\mu},\boldsymbol{\Sigma})
\mapsto
\mathbf p^u\)
defined by \eqref{p_formula} is real analytic on $\RR^2\times\mathrm{PD}_2$ \footnote{\(\mathrm{PD}_k\) denotes the cone of \(k\times k\) positive definite matrices.}.

In the next subsection, we will show that the marginal probabilities \(P(\mathbf y\succeq\mathbf r)\) can be represented as integrals involving multivariate normal densities. Since multivariate normal densities are real analytic, the following proposition implies that these marginal probabilities are also real analytic functions of the distribution parameters.

In the next subsection, we will show that the marginal probabilities $P(\mathbf y\succeq\mathbf r)$ can be expressed as multivariate normal CDFs evaluated at real analytic transformations of $(\boldsymbol{\mu},\Sigma)$. Since multivariate normal densities are real analytic, the following proposition implies that these marginal probabilities are also real analytic functions of $(\boldsymbol{\mu},\Sigma)$.

\begin{proposition}[Proposition 2.2.3 in \cite{krantz2002primer}]
Let $f$ be a real analytic function defined on an open subset $U\subseteq\RR^m$. Then $f$ is continuous and all of its partial derivatives exist and are real analytic. Furthermore, the indefinite integral of $f$ with respect to any variable is real analytic.
\end{proposition}

Therefore, the Jacobian matrix has full column rank almost everywhere on $\Omega'$, establishing local identifiability of the submodel in Example~\ref{eg_5items}. The case $K=1$ and $J=4$, corresponding to case (i) of Theorem~\ref{main_theorem}, is verified similarly.

\subsection{Step 2: From \(\mathbf{p}\) to \((\boldsymbol{\mu},\Sigma)\)}\label{mu_and_sigma}
When each latent attribute is measured by at least three items depending only on that attribute, for every pair of latent attributes \((k,k')\), we can select five items whose $Q$-matrix has the same structure as in Example~\ref{eg_5items}. Each such submodel identifies the corresponding Gaussian parameters \((\boldsymbol{\mu}_{k,k'},\Sigma_{kk',kk'})\). By combining these $K(K-1)/2$ overlapping local structures, the full Gaussian parameters \((\boldsymbol{\mu},\Sigma)\) can be recovered.

Next, we show how to identify \((\boldsymbol{\mu}_{1,2},\Sigma_{12,12})\) from the corresponding reduced latent class probabilities $\mathbf p^u\in\Delta$. 
To this end, we reformulate the integral representation in \eqref{p_formula} into an equivalent system, which falls under a special case of Proposition~2 in \cite{fang2021identifiability}. In this subsection, we focus on the local Gaussian structure \((\boldsymbol{\mu}_{1,2}, {\Sigma}_{12,12})\) and the corresponding latent subvector $\mathbf d_{1,2}$. For notational simplicity, we suppress the subscripts and write \((\boldsymbol{\mu},{\Sigma})\) and $\mathbf d$ instead.

Recall that $o_j$ denotes the realized binary indicator for item $j$ in the block.
For the latent mastery score \(\mathbf d\sim N(\boldsymbol{\mu},\boldsymbol{\Sigma})\),
the marginal probability of $o_1=1$ is
\begin{align*}
P(o_1=1)
&=
\EE_{\mathbf d}[P(o_1=1\mid\mathbf d)]\\
&=
\EE_{\mathbf d'}\!\left[
P(\epsilon_1\le d_1'+\mu_1\mid \mathbf{d}')
\right]
\qquad
(\mathbf d'=\mathbf d-\boldsymbol{\mu})\\
&=
P\!\left(
\mu_1+\sqrt{\sigma_{11}+1}\,\xi\ge0
\right)\\
&=
1-\Phi\!\left(
-\frac{\mu_1}{\sqrt{\sigma_{11}+1}}
\right),
\end{align*}
where \(\xi\sim N(0,1)\).

The marginal probability of $o_1=1,o_2=1$ is
\begin{align*}
P(o_1=1, o_2=1) 
&= \EE_{\mathbf{d}}[P(o_1=1, o_2=1\mid \mathbf{d})]\\
&=\EE\left[P(\epsilon_1\leq d_1'+\mu_1,\ \epsilon_2\leq d_1'+\mu_1\mid \mathbf d')\right] \qquad
(\mathbf d'=\mathbf d-\boldsymbol{\mu})\\
&= P(\mu_1+\sqrt{\sigma_{11}+1}\,\xi_1\geq 0, \mu_1+\sqrt{\sigma_{11}+1}\,\xi_2\geq 0)\\
&=\Phi_2\left(-\frac{\mu_1}{\sqrt{\sigma_{11}+1}}, -\frac{\mu_1}{\sqrt{\sigma_{11}+1}}, \frac{\sigma_{11}}{{\sigma_{11}+1}}\right),
\end{align*}
where \(\epsilon_1,\epsilon_2\sim N(0,1)\) and independent, 
\[
\xi_1=\frac{d_1'+\epsilon_1}{\sqrt{\sigma_{11}+1}},
\qquad
\xi_2=\frac{d_1'+\epsilon_2}{\sqrt{\sigma_{11}+1}}.
\]
Then \((\xi_1,\xi_2)\) is bivariate standard normal with correlation
\[
\rho=\frac{\sigma_{11}}{\sigma_{11}+1}.
\]
Here \(\Phi_2(a,b,\rho)\) denotes the upper-tail probability
$
P(X_1\ge a,\ X_2\ge b),
$
where $(X_1,X_2)$ follows the bivariate normal distribution
\[
N\left(
\begin{pmatrix}
0\\
0
\end{pmatrix},
\begin{pmatrix}
1&\rho\\
\rho&1
\end{pmatrix}
\right).
\]

Similarly, we obtain
\[
P(o_4=1)=1-\Phi\left(-\frac{\mu_2}{\sqrt{\sigma_{22}+1}}\right),
\]
\[
P(o_4=1, o_5=1)=\Phi_2\left(-\frac{\mu_1}{\sqrt{\sigma_{22}+1}}, -\frac{\mu_2}{\sqrt{\sigma_{22}+1}}, \frac{\sigma_{22}}{{\sigma_{22}+1}}\right),
\]
and
\[
P(o_1=1, o_4=1)=\Phi_2\left(-\frac{\mu_1}{\sqrt{\sigma_{11}+1}}, -\frac{\mu_2}{\sqrt{\sigma_{22}+1}}, \frac{\sigma_{12}}{\sqrt{\sigma_{11}+1}\sqrt{\sigma_{22}+1}}\right).
\]
The marginal probabilities can be expressed in terms of the reduced latent class probabilities $\mathbf p^u$:
\begin{align*}
P(o_1=1) &= p^u_{10} + 2p^u_{20} + p^u_{30} + 2\cdot(p^u_{11} + 2p^u_{21} + p^u_{31}) + p^u_{12} + 2p^u_{22} + p^u_{32},\\
P(o_4=1) &= p^u_{01} + 3\cdot (p^u_{11} + p^u_{21}) + p^u_{31} + p^u_{02} + 3\cdot(p^u_{12} + p^u_{22}) + p^u_{32},\\
P(o_1=1,o_2=1) &= p^u_{20} + p^u_{30} + 2\cdot(p^u_{21} + p^u_{31}) + p^u_{22} + p^u_{32},\\ 
P(o_1=1,o_4=1) &= p^u_{11} + 2p^u_{21} + p^u_{31} + p^u_{12} + 2p^u_{22} + p^u_{32},\\
P(o_4=1,o_5=1) &= p^u_{02} + 3p^u_{12} + 3p^u_{22} + p^u_{32}. 
\end{align*}
For a fixed $\mathbf{p}^u$, the parameters $\mu_1$ and $\sigma_{11}$ are identifiable. Indeed, suppose there exist two distinct pairs \((\mu_1,\sigma_{11})\) and \((\overline{\mu}_1,\overline{\sigma}_{11})\) such that
\begin{align*}
1-\Phi\left(-\frac{\mu_1}{\sqrt{\sigma_{11}+1}}\right)&=1-\Phi\left(-\frac{\overline{\mu}_1}{\sqrt{\overline{\sigma}_{11}+1}}\right),\\
\Phi_2\left(-\frac{\mu_1}{\sqrt{\sigma_{11}+1}}, -\frac{\mu_1}{\sqrt{\sigma_{11}+1}}, \frac{\sigma_{11}}{{\sigma_{11}+1}}\right)& = \Phi_2\left(-\frac{\overline{\mu}_1}{\sqrt{\overline{\sigma}_{11}+1}}, -\frac{\overline{\mu}_1}{\sqrt{\overline{\sigma}_{11}+1}}, \frac{\overline{\sigma}_{11}}{{\overline{\sigma}_{11}+1}}\right),
\end{align*}
which implies
\begin{align*}
\frac{\mu_1}{\sqrt{\sigma_{11}+1}} &=    \frac{\overline{\mu}_1}{\sqrt{\overline{\sigma}_{11}+1}},\\
\frac{\sigma_{11}}{{\sigma_{11}+1}}&=\frac{\overline{\sigma}_{11}}{{\overline{\sigma}_{11}+1}}.
\end{align*}
The second equation implies that $\sigma_{11}= \overline{\sigma}_{11}$. Substituting this into the first equation yields $\mu_{1}=\overline{\mu}_1$. 

The same arguments applies to $\mu_2,\sigma_{22}$ using the formulas of $P(o_4=1)$ and $P(o_4=1, o_5=1)$. Finally, $\sigma_{12}$ can be determined by the formula of $P(o_1=1, o_4=1)$. Therefore, the full Gaussian parameters \((\boldsymbol{\mu},\Sigma)\) are identifiable. 

\subsection{Step 3: Identification of the Remaining Item Parameters}

The final step is to identify the item parameters of the remaining $J-3K$ items. For each such item $j$, we select $K$ previously identified items \((j_1,\dots,j_K)\). The only unknown quantities are the item parameters associated with item $j$. The joint response probabilities of \((j_1,\dots,j_K,j)\) yield a system of $2^K$ equations, whose coefficient matrix is determined by the already identified item parameters and the identified copula parameters \((\boldsymbol{\mu},\Sigma)\).

\subsubsection{PM-DINA Model}
In a PM-DINA model, one previously identified item together with two response equations is sufficient to identify $\theta_{j0}$ and $\theta_{j1}$. For any $j>3K$, we consider the joint distribution of item 1 and $j$, together with the corresponding binary attribute realizations of length $2$. The associated reduced latent class probabilities and item parameters are summarized in the following matrix.
\begin{align}
\begin{blockarray}{cccc}
00 & 10 & 01 & 11 \\
p_{00} & p_{10} & p_{01} & p_{11}\\
\begin{block}{(cccc)}
\theta_{10} & \theta_{11} & \theta_{10} & \theta_{11}\\
\theta_{j0} & \theta_{j0} & \theta_{j1} & \theta_{j1}\\
\end{block}
\end{blockarray}.
\end{align}
There are two equations for the two parameters $\theta_{j0},\theta_{j1}$:
\begin{align}
(p_{00}+p_{10})\cdot\theta_{j0} + (p_{01}+p_{11})\cdot\theta_{j1} &= P(Y_j=1),\\
(p_{00}\theta_{10}+p_{10}\theta_{11})\cdot \theta_{j0} + (p_{01}\theta_{10}+p_{11}\theta_{11})\cdot \theta_{j1} &= P(Y_1=1, Y_j=1).
\end{align}
When 
\[
(p_{00}+p_{10})\cdot (p_{01}\theta_{10}+p_{11}\theta_{11}) - (p_{01}+p_{11})\cdot (p_{00}\theta_{10}+p_{10}\theta_{11})\neq 0,
\]
the system admits a unique solution for $\theta_{j0}$ and $\theta_{j1}$. The condition can be reduced to
\[
(p_{00}p_{11}-p_{01}p_{10})\cdot(\theta_{11}-\theta_{10})\neq 0.
\]
We know $\theta_{11}-\theta_{10}>0$ from model assumption. When item $j$ and item $1$ share a required latent attribute, by the dependence structure $y_1\leftarrow \alpha_k\rightarrow y_j$, 
the corresponding realized binary indicators $o_1$ and $o_j$ are generally dependent. Therefore,
\[
p_{00}p_{11}-p_{01}p_{10}\neq 0.
\]

\subsubsection{PM-GDINA Model}
In a PM-GDINA model, we need $K$ previously identified items together with $2^K$ equations to identify the $2^K$ parameters
\[
\{\theta_{j,\boldsymbol{\beta}}:\boldsymbol{\beta}\in\{0,1\}^K\}.
\]
We assume that the first $K$ rows of $Q$ form the identity matrix $I_K$, whose item parameters are already identified from previous analysis.  The reduced latent class probabilities for the joint structure of items $1,\dots,K$ and item $j$ contain $2^K\times 2^K=2^{2K}$ distinct values, denoted by
$p_{\boldsymbol{\alpha},\boldsymbol{\beta}}$, $\boldsymbol{\alpha}, \boldsymbol{\beta}\in\{0,1\}^K$. 
Here $\boldsymbol{\alpha}$ indexes the realized attribute pattern for items $1,\dots,K$, while $\boldsymbol{\beta}$ indexes the realized attribute pattern associated with item $j$. The corresponding matrix of marginal probabilities
\((t_{\mathbf{e}_k,\boldsymbol{\alpha}\boldsymbol{\beta}}(Q,\Theta, \mathbf{p}))\) has dimension $(K+1)\times 2^{2K}$. 

For a specific index $\boldsymbol{\alpha}$, we define the vector $\boldsymbol{\theta}^*=(\theta_{1,(1-\alpha_1)}, \cdots, \theta_{K,(1-\alpha_K)}, 0)^T$ and consider the matrix $T_{I_{K+1},\cdot}(Q,\Theta-\theta^*\mathbf{1}^T)$. For any $\boldsymbol{\alpha}'\neq\boldsymbol{\alpha}$ and any $\boldsymbol{\beta}$, there exists some $k$ such that $\alpha_k'=1-\alpha_k$. Hence,
\[
t_{\mathbf{e}_k,\boldsymbol{\alpha}'\boldsymbol{\beta}}(Q, \Theta)=\theta_{k,\alpha_k'}=\theta_{k,(1-\alpha_k)},
\]
and
\begin{align*}
t_{\mathbf{e}_k,\boldsymbol{\alpha}'\boldsymbol{\beta}}(Q,\Theta-\boldsymbol{\theta}^*\mathbf{1}^T)&=0,\\
t_{\mathbf{e}_k,\boldsymbol{\alpha}\boldsymbol{\beta}}(Q,\Theta-\boldsymbol{\theta}^*\mathbf{1}^T)&=(-1)^{\alpha_k}(\theta_{k0}-\theta_{k1}).
\end{align*}
The column indexed by $p^u_{\boldsymbol{\alpha}'\boldsymbol{\beta}}$ has at least one zero entry among the first $K$ rows, whereas every entry in the column indexed by $p^u_{\boldsymbol{\alpha}\boldsymbol{\beta}}$ is nonzero. Moreover, the first $K$ entries in the column indexed by $p^u_{\boldsymbol{\alpha}\boldsymbol{\beta}}$ depend only on $\boldsymbol{\alpha}$. By the definition of the $T$-matrix, we have
\begin{align*}
T_{\sum_{k=1}^K\mathbf{e}_k+\mathbf{e}_j, \cdot}(Q,\Theta-\boldsymbol{\theta}^*\mathbf{1}^T)=\bigodot_{k\in\{1,\cdots,K,j\}}T_{\mathbf{e}_k, \cdot}(Q,\Theta-\boldsymbol{\theta}^*\mathbf{1}^T),
\end{align*}
in which
\begin{align*}
t_{\sum_{k=1}^K\mathbf{e}_k+\mathbf{e}_j,\boldsymbol{\alpha}'\boldsymbol{\beta}}(Q,\Theta-\boldsymbol{\theta}^*\mathbf{1}^T)&=0,\\
t_{\sum_{k=1}^K\mathbf{e}_k+\mathbf{e}_j,\boldsymbol{\alpha}\boldsymbol{\beta}}(Q,\Theta-\boldsymbol{\theta}^*\mathbf{1}^T)&=(-1)^{\sum_{k=1}^K\alpha_k}\prod_{k=1}^K (\theta_{k0}-\theta_{k1}).
\end{align*}
This implies that
\begin{align}\label{T}
T_{\sum_{k=1}^K\mathbf{e}_k+\mathbf{e}_j,\cdot}(Q,\Theta-\boldsymbol{\theta}^*\mathbf{1}^T)\mathbf{p}=\left((-1)^{\sum_{k=1}^K\alpha_k}\prod_{k=1}^K (\theta_{k0}-\theta_{k1})\right)\cdot \sum_{\boldsymbol{\beta}\in\{0,1\}^K}p_{\boldsymbol{\alpha\boldsymbol{\beta}}}^u \theta_{j,\boldsymbol{\beta}}.
\end{align}
By Proposition \ref{prop3_xu}, we have 
\begin{align}\label{T_minus}
T_{\sum_{k=1}^K\mathbf{e}_k+\mathbf{e}_j,\cdot}(Q,\Theta-\boldsymbol{\theta}^*\mathbf{1}^T)\mathbf{p} &= \Big(D(\boldsymbol{\theta}^*)T(Q,\Theta)\mathbf{p}\Big)_{\sum_{k=1}^K\mathbf{e}_k+\mathbf{e}_j}\nonumber \\
&=\Big(D(\boldsymbol{\theta}^*)\cdot \big(P(\mathbf{R}\succeq \mathbf{r}\mid Q,\Theta, \mathbf{p})\big)_{\mathbf{r}\in\{0,1\}^{K+1}}^T\Big)_{\sum_{k=1}^K\mathbf{e}_k+\mathbf{e}_j}.
\end{align}
The right-hand side of \eqref{T_minus} depends only on the item parameters of item $1,\cdots, K$ and the corresponding marginal probability \(P(\mathbf{R}_{\{1,\cdots,K,j\}}\succeq \mathbf{r},\mathbf{r}\in\{0,1\}^{K+1})\). Combining \eqref{T} and \eqref{T_minus}, we obtain a linear equation in \(\{\theta_{j,\boldsymbol{\beta}}:\boldsymbol{\beta}\in\{0,1\}^K\}\):
\begin{align}\label{eq_theta_j_alpha}
\sum_{\boldsymbol{\beta}\in\{0,1\}^K}p_{\boldsymbol{\alpha\boldsymbol{\beta}}} \theta_{j,\boldsymbol{\beta}} = C_{\boldsymbol{\alpha}},
\end{align}
where $C_{\balpha}$ is a constant depending only on $\boldsymbol{\alpha}$, $\Theta_{1:K,\cdot}$, and the marginal probabilities.
Varying $\boldsymbol{\alpha}$ over $\{0,1\}^K$ yields $2^K$ linear equations of the form \eqref{eq_theta_j_alpha}, resulting a linear system
\begin{align}\label{eq_theta_j_all}
\mathbf{P} \cdot (\theta_{j,\boldsymbol{\beta}})_{\boldsymbol{\beta}\in\{0,1\}^K}^T = (C_{\boldsymbol{\alpha}})_{\boldsymbol{\alpha}\in\{0,1\}^K}^T.
\end{align}
To establish identifiability of $(\theta_{j,\boldsymbol{\beta}})_{\boldsymbol{\beta}\in\{0,1\}^K}$, it suffices to show that
$\det(\mathbf{P}^u)\neq 0$.

The $(\boldsymbol{\alpha},\boldsymbol{\beta})$ entry of $\mathbf{P}$ is 
\begin{align}
p_{\boldsymbol{\alpha}\boldsymbol{\beta}}=\int_{\RR^K}\prod_{k=1}^K\Phi(z_k)^{\alpha_k+\beta_k}(1-\Phi(z_k))^{2-\alpha_k-\beta_k}dN(\mathbf{z}\mid \boldsymbol{\mu}, \boldsymbol{\Sigma}).
\end{align}
Let $\nu=N(\mathbf{z}\mid \boldsymbol{\mu}, \boldsymbol{\Sigma})$ be a probability measure on $\RR^K$. The function $f_{\boldsymbol{\alpha}}(\mathbf{z})=\prod_{k=1}^K\Phi(z_k)^{\alpha_k}(1-\Phi(z_k))^{1-\alpha_k}$ is uniformly bounded on $\RR^K$ and thus $f_{\boldsymbol{\alpha}}\in \mathbb{L}^1_\nu(\RR^K) \bigcap \mathbb{L}^2_\nu(\RR^K)$ \footnote{\(\mathbb L^p_\nu(\RR^K)\) denotes the space of measurable functions \(f\) satisfying \(\int_{\RR^K} |f|^p\,d\nu<\infty\).}. We can define the inner product on $\mathbb{L}^2_\nu(\RR^K)$ induced by $\nu$: 
\begin{align*}
\langle f,g \rangle_\nu=\int_{\RR^K}f\cdot g \ d\nu.    
\end{align*}
The coefficient matrix $\mathbf{P}^u$ is the Gram matrix of the family $\{f_{\boldsymbol{\alpha}}:\boldsymbol{\alpha}\in\{0,1\}^K\}$. $\mathbf{P}^u$ is positive definite if and only if this family of functions is linearly independent. That is, there do not exist constants $c_{\boldsymbol{\alpha}}, \boldsymbol{\alpha}\in\{0,1\}^K$, not all zero, such that 
\begin{align}\label{lindep}
\sum_{\boldsymbol{\alpha}\in\{0,1\}^K}c_{\boldsymbol{\alpha}}f_{\boldsymbol{\alpha}} = 0.
\end{align}
We prove this by contradiction. Suppose that \eqref{lindep} holds for some coefficients $c_{\boldsymbol{\alpha}}, \boldsymbol{\alpha}\in\{0,1\}^K$, not all zero. We can then choose $2^K$ points $\mathbf{x}_{\boldsymbol{\alpha}'}$ such that $\det(f_{\boldsymbol{\alpha}}(\mathbf{x}_{\boldsymbol{\alpha}'}))\neq 0$. The resulting linear system
\begin{align*}
\sum_{\boldsymbol{\alpha}\in\{0,1\}^K}f_{\boldsymbol{\alpha}}(\mathbf{x}_{\boldsymbol{\alpha}'})\cdot c_{\boldsymbol{\alpha}} = 0, \quad \boldsymbol{\alpha}'\in\{0,1\}^K
\end{align*}
has only the trivial solution.

Choose $\varepsilon>0$ such that $(1-\varepsilon)^{4^K}-(2^K)!\cdot \varepsilon > 0$, and define $C_\varepsilon=\Phi^{-1}(1-\varepsilon)=-\Phi^{-1}(\varepsilon)$. The $2^K$ points are $\mathbf{x}_{\boldsymbol{\alpha}'}=C_\varepsilon\cdot (2\boldsymbol{\alpha}'-1)$. Then we have
\begin{align*}
f_{\boldsymbol{\alpha}}(\mathbf{x}_{\boldsymbol{\alpha}})=\prod_{k=1}^K\Phi(C_\varepsilon)^{\alpha_k}(1-\Phi(-C_\varepsilon))^{1-\alpha_k}=(1-\varepsilon)^{2^K}.
\end{align*}
For $\boldsymbol{\alpha}'\neq \boldsymbol{\alpha}$, there exists some $k$ such that $\alpha_k\neq \alpha'_k$, and
\begin{align*}
f_{\boldsymbol{\alpha}}(\mathbf{x}_{\boldsymbol{\alpha}'})
&\leq \Phi(C_\varepsilon(2\alpha_k'-1))^{\alpha_k}(1-\Phi(-C_\varepsilon(2\alpha_k'-1)))^{1-\alpha_k}\\
&=\Phi(-C_\varepsilon)^{\alpha_k}(1-\Phi(C_\varepsilon))^{1-\alpha_k}\\
&=\varepsilon.
\end{align*}
The $2^K\times 2^K$ matrix $f_{\boldsymbol{\alpha}}(\mathbf{x}_{\boldsymbol{\alpha}'})$ is a nearly diagonal matrix after row and column permutations, with diagonal entries equal to $(1-\varepsilon)^{2^K}$ and off-diagonal entries lying in $(0,\varepsilon)$. It follows that
\begin{align*}
\big|\det(f_{\boldsymbol{\alpha}}(\mathbf{x}_{\boldsymbol{\alpha}'}))\big|\geq \big((1-\varepsilon)^{2^K}\big)^{2^K} - (2^K)!\cdot \varepsilon=(1-\varepsilon)^{4^K}-(2^K)!\cdot \varepsilon > 0.
\end{align*}
Hence,
$c_{\boldsymbol{\alpha}}=0$ for all $\boldsymbol{\alpha}\in\{0,1\}^K$. This implies that the family $\{f_{\boldsymbol{\alpha}}:\boldsymbol{\alpha}\in\{0,1\}^K\}$ is linearly independent in $\mathbb{L}^2_\nu(\RR^K)$, and hence its Gram matrix $\mathbf{P}^u$ is positive definite. Therefore, the linear equation system \eqref{eq_theta_j_all} admits a unique solution for $\theta_{j,\boldsymbol{\beta}},\boldsymbol{\beta}\in\{0,1\}^K$.
\qed

To provide intuition for the construction of linear equations by Proposition~\ref{prop3_xu}, we consider a simple example with $K=2$. The following matrix is $T(Q,\Theta)$.
\begin{align*}
\begin{blockarray}{cccccccc}
p_{00,00} & p_{10,00} & p_{01,00} & p_{11,00} & p_{00,01} & p_{10,01} & p_{01,01} & p_{11,01}\\
\begin{block}{(cccccccc)}
\theta_{10} & \theta_{11} & \theta_{10} & \theta_{11} & \theta_{10} & \theta_{11} & \theta_{10} & \theta_{11}\\
\theta_{20} & \theta_{20} & \theta_{21} & \theta_{21} & 
\theta_{20} & \theta_{20} & \theta_{21} & \theta_{21}\\
\theta_{j,00} & \theta_{j,00} & \theta_{j,00} & \theta_{j,00} & 
\theta_{j,01} & \theta_{j,01} & \theta_{j,01} & \theta_{j,01}\\
\end{block}
\end{blockarray}
\end{align*}
\begin{align*}
\begin{blockarray}{cccccccc}
p_{00,10} & p_{10,10} & p_{01,10} & p_{11,10} & p_{00,11} & p_{10,11} & p_{01,11} & p_{11,11}\\
\begin{block}{(cccccccc)}
\theta_{10} & \theta_{11} & \theta_{10} & \theta_{11} & \theta_{10} & \theta_{11} & \theta_{10} & \theta_{11}\\
\theta_{20} & \theta_{20} & \theta_{21} & \theta_{21} & 
\theta_{20} & \theta_{20} & \theta_{21} & \theta_{21}\\
\theta_{j,10} & \theta_{j,10} & \theta_{j,10} & \theta_{j,10} & 
\theta_{j,11} & \theta_{j,11} & \theta_{j,11} & \theta_{j,11}\\
\end{block}
\end{blockarray}.
\end{align*}
Let $\boldsymbol{\alpha}=(01)$ and $\boldsymbol{\theta}^*=(\theta_{11},\theta_{20},0)^T$, the transformed $T$-matrix $T(Q,\Theta-\boldsymbol{\theta}^*\mathbf{1}^T)$ is
\begin{align}
\begin{blockarray}{cccccccc}
p_{00,00} & p_{10,00} & p_{01,00} & p_{11,00} & p_{00,01} & p_{10,01} & p_{01,01} & p_{11,01}\\
\begin{block}{(cccccccc)}
\theta_{10}-\theta_{11} & 0 & \theta_{10}-\theta_{11} & 0 & \theta_{10}-\theta_{11} & 0 & \theta_{10}-\theta_{11} & 0\\
0 & 0 & \theta_{21}-\theta_{20} & \theta_{21}-\theta_{20} & 
0 & 0 & \theta_{21}-\theta_{20} & \theta_{21}-\theta_{20}\\
\theta_{j,00} & \theta_{j,00} & \theta_{j,00} & \theta_{j,00} & 
\theta_{j,01} & \theta_{j,01} & \theta_{j,01} & \theta_{j,01}\\
\end{block}
\end{blockarray},\nonumber \\
\begin{blockarray}{cccccccc}
p_{00,10} & p_{10,10} & p_{01,10} & p_{11,10} & p_{00,11} & p_{10,11} & p_{01,11} & p_{11,11}\\
\begin{block}{(cccccccc)}
\theta_{10}-\theta_{11} & 0 & \theta_{10}-\theta_{11} & 0 & \theta_{10}-\theta_{11} & 0 & \theta_{10}-\theta_{11} & 0\\
0 & 0 & \theta_{21}-\theta_{20} & \theta_{21}-\theta_{20} & 
0 & 0 & \theta_{21}-\theta_{20} & \theta_{21}-\theta_{20}\\
\theta_{j,10} & \theta_{j,10} & \theta_{j,10} & \theta_{j,10} & 
\theta_{j,11} & \theta_{j,11} & \theta_{j,11} & \theta_{j,11}\\
\end{block}
\end{blockarray}.
\end{align}

The choice of $\balpha$ corresponds to the equation 
\begin{align*}
\Big(D(\boldsymbol{\theta}^*)\cdot \big(P(\mathbf{R}\succeq \mathbf{r}&\mid Q,\Theta, \mathbf{p})\big)_{\mathbf{r}\in\{0,1\}^{3}}^T\Big)_{\sum_{k=1}^2\mathbf{e}_k+\mathbf{e}_j}\\
&= \left(t_{\mathbf{e}_1,\cdot}(Q, \Theta-\boldsymbol{\theta}^*\mathbf{1}^T)\odot t_{\mathbf{e}_2,\cdot}(Q, \Theta-\boldsymbol{\theta}^*\mathbf{1}^T)\odot t_{\mathbf{e}_j,\cdot}(Q, \Theta-\boldsymbol{\theta}^*\mathbf{1}^T) \right)^T \mathbf{p} \\
&= -(\theta_{10}-\theta_{11})(\theta_{20}-\theta_{21})\cdot (p^u_{01,00} \theta_{j,00} + p^u_{01,01} \theta_{j,01} + p^u_{01,10} \theta_{j,10} + p^u_{01,11} \theta_{j,11}).
\end{align*}
The cases $\balpha=(00),(10),(11)$ are handled similarly, together yielding four equations for $\theta_{j,00}, \theta_{j,01}, \theta_{j,10}$, and $\theta_{j,11}$.

\subsection{A Relaxation of the Structural Condition}
The proof of Theorem~\ref{main_theorem} does not fully exploit the information obtained in Step~1. In Step~2, we identify $(\boldsymbol{\mu},\Sigma)$ by considering every pair of latent attributes separately. However, once the item parameters have been identified in Step~1, it is unnecessary to construct same local models for all attribute pairs.

Indeed, it suffices to use only $K$ local structures to identify $\boldsymbol{\mu}$ and the diagonal entries of $\Sigma$.
The remaining off-diagonal entries can then be recovered from simpler local structures involving the corresponding attribute pairs.
Consequently, the structural condition in Theorem~\ref{main_theorem} can be relaxed: it is sufficient that one latent attribute be measured by at least three pure items, while every other latent attribute be measured by at least two pure items.

\begin{theorem}\label{relaxed_theorem}
The conclusion of Theorem~\ref{main_theorem}
continues to hold if $K\ge 2$ and the $Q$-matrix takes the form
\[
\begin{pmatrix}
I_K\\
I_K\\
Q'
\end{pmatrix},
\]
where $Q'$ contains a row equal to $\mathbf e_k$ for some $k\in\{1,\dots,K\}$.
\end{theorem}
\begin{proof}
Without loss of generality, , assume that $Q'$ contains the row $\mathbf e_1$. The arguments in Steps~1 and~2 can then be applied to the latent attribute pairs $(1,2),\dots,(1,K)$.
Consequently, the item parameters for items $1,\ldots,2K$, the copula mean vector $\boldsymbol{\mu}$, and the diagonal entries of $\Sigma$ are identifiable.

To identify an off-diagonal entry $\sigma_{kk'}$, consider the joint distribution of items $k$ and $k'$, together with the corresponding binary attribute realizations of length two. Since item $k$ requires only latent attribute $k$ and item $k'$ requires only latent attribute $k'$, the associated local structure depends only on the bivariate latent distribution of $(d_k,d_{k'})$. The corresponding $T$-matrix $T(I_2, \Theta)$ is of size $4\times 4$. Two of its rows, $t_{10,\cdot}$ and $t_{01,\cdot}$, are displayed in the following matrix.
\begin{align*}
\begin{blockarray}{cccc}
00 & 10 & 01 & 11 \\
p_{00} & p_{10} & p_{01} & p_{11}\\
\begin{block}{(cccc)}
\theta_{k0} & \theta_{k1} & \theta_{k0} & \theta_{k1}\\
\theta_{k'0} & \theta_{k'0} & \theta_{k'1} & \theta_{k'1}\\
\end{block}
\end{blockarray}.
\end{align*}
Choose $\boldsymbol{\theta}^*=(\theta_{k0}, \theta_{k'0})^T$. 
By Proposition~\ref{prop3_xu}, the vector $T(I_2, \Theta-\boldsymbol{\theta}^*\mathbf{1})\mathbf{p}$ is uniquely determined by the joint distribution of $(y_k, y_{k'})$. Notice that
\begin{align*}
T_{11,\cdot}(I_2,\Theta-\boldsymbol{\theta}^*\mathbf{1})\mathbf{p}&=\Big(T_{10,\cdot}(I_2,\Theta-\boldsymbol{\theta}^*\mathbf{1}) \odot T_{01,\cdot}(I_2,\Theta-\boldsymbol{\theta}^*\mathbf{1})\Big)\cdot\mathbf{p}\\
&=\Big((0, 0, \theta_{k1}-\theta_{k0}, \theta_{k1}-\theta_{k0})\odot (0, \theta_{k'1}-\theta_{k'0}, 0, \theta_{k'1}-\theta_{k'0})\Big) \cdot \mathbf p\\
&=\Big(0,0,0,(\theta_{k1}-\theta_{k0})(\theta_{k'1}-\theta_{k'0})\Big)\cdot \mathbf p\\
&= (\theta_{k1}-\theta_{k0})(\theta_{k'1}-\theta_{k'0})p_{11}.
\end{align*}
Since $\theta_{k1}-\theta_{k0}>0,\,\theta_{k'1}-\theta_{k'0}>0$, it follows that $p_{11}$ is uniquely identified.
Use the same arguments as in Step 2, we know that
\[
p_{11}=P(o_k=1,o_{k'}=1)=\Phi_2\left(-\frac{\mu_k}{\sqrt{\sigma_{kk}+1}}, -\frac{\mu_k}{\sqrt{\sigma_{k'k'}+1}}, \frac{\sigma_{kk'}}{\sqrt{\sigma_{kk}+1}\sqrt{\sigma_{k'k'}+1}}\right).
\]
Since $\mu_{k},\mu_{k'},\sigma_{kk},\sigma_{k'k'}$ have already been identified, the reduced latent class probability $p_{11}$ uniquely determines the off-diagonal entry $\sigma_{kk'}$.
Hence, $\sigma_{kk'}$ is identifiable.

The same argument applies to every pair of latent attributes. Therefore, all off-diagonal entries of $\Sigma$ are identifiable. The remainder of the proof is identical to Step~3 in the proof of Theorem~\ref{main_theorem}.
\end{proof}



\section{Applications}

\subsection{Extending Identifiability Results to Additive PM-CDMs}
The additive partial-mastery CDMs (aPM-CDMs) are proposed by \citet{cardenas2025generalized}, bridging the gap between PM-DINA and PM-GDINA models. The response probability given latent mastery score $\mathbf d$ is
\begin{equation}\label{eq_aPM-CDM}
\theta_{j,\mathbf{d}}=\delta_{j0} + \sum^K_{k=1}\delta_{jk}q_{jk} d_k, \qquad \delta_{jk}\geq 0,\,\delta_{j0}+\sum^K_{k=1}\delta_{jk}q_{jk}\leq 1.
\end{equation}
We next show that both PM-DINA and aPM-CDMs can be viewed as constrained PM-GDINA models through a reformulation of the PM-GDINA response probability. Consequently, the local identifiability result established for PM-GDINA also applies to aPM-CDMs.

To relate aPM-CDMs to PM-GDINA, we rewrite the PM-GDINA response function. Recall that
\begin{align}\label{theta_jd_expand}
\theta_{j,\mathbf d}
=
\sum_{\boldsymbol{\alpha}\in\{0,1\}^K}
\theta_{j,\boldsymbol{\alpha}}
\prod_{k=1}^K
d_k^{\alpha_k}(1-d_k)^{1-\alpha_k}.
\end{align}
Expanding \eqref{theta_jd_expand} in the interaction basis
\[
\phi(\mathbf d)=
\Bigl\{
\prod_{k\in S} d_k:
S\subseteq\{1,\ldots,K\}
\Bigr\}
\]
and collecting coefficients yields
\begin{align}
\theta_{j,\mathbf d}
=
\phi(\mathbf d)^T\cdot (M\boldsymbol{\theta}_j),
\label{PMCDM-irf}
\end{align}
where $\boldsymbol{\theta}_j$ is the $2^K$-dimensional vector of item parameters for item $j$, and $M$ is a $2^K\times2^K$ transformation matrix. The vector $M\boldsymbol{\theta}_j$
contains the coefficients associated with the interaction basis
$\phi(\mathbf d)$, whose coordinates are ordered according to the binary encoding of the corresponding subsets.

For example, when $K=2$, the ordering is
$00,01,10,11$, which corresponds to $\varnothing,\{2\},\{1\},\{1,2\}$. The response probability is
\begin{equation*}\label{expansion_K=2}
\begin{aligned}
\theta_{j,\mathbf{d}}&=
\boldsymbol{\theta}_j^T\cdot\mathbf{p}_{\balpha\mid \mathbf d} \\
&=\theta_{j,00}p_{00\mid \mathbf d}+\theta_{j,01}p_{01\mid \mathbf d}+\theta_{j,10}p_{10\mid \mathbf d}+\theta_{j,11}p_{11\mid\mathbf d}\\
&=\theta_{j,00}(1-d_1)(1-d_2) + \theta_{j,01}(1-d_1)d_2 + \theta_{j,10}d_1(1-d_2)+\theta_{j,11}d_1d_2\\
&=\theta_{j,00} + (\theta_{j,01}-\theta_{j,00}){d_2} + (\theta_{j,10}-\theta_{j,00}){d_1} + (\theta_{j,00}-\theta_{j,01}-\theta_{j,10}+\theta_{j,11}){d_1d_2}\\
&=(M\boldsymbol{\theta}_j)_{1}\cdot 1 + (M\boldsymbol{\theta}_j)_{2}\cdot d_2 +(M\boldsymbol{\theta}_j)_{3}\cdot d_1 +
(M\boldsymbol{\theta}_j)_{4}\cdot d_1d_2\\
&= (1, d_1, d_2, d_1d_2)\cdot (M\boldsymbol{\theta}_j).
\end{aligned}
\end{equation*}
This yields the transformation matrix
\[
M=
\begin{pmatrix}
1 & 0 & 0 & 0\\
-1 & 1 & 0 & 0\\
-1 & 0 & 1 & 0\\ 
1 & -1 & -1 & 1
\end{pmatrix}.
\]


The following lemma extends the above construction to arbitrary $K$.
\begin{lemma}\label{theta_d_repre}
The response probability of a PM-GDINA model, given the latent mastery score $\mathbf d$, can be written as
\[
\theta_{j,\mathbf d}
=
\phi(\mathbf d)^T\cdot(M\boldsymbol{\theta}_j),
\]
where the rows and columns of $M$ are indexed by
$\boldsymbol{\alpha}\in\{0,1\}^K$
in binary order, and
\[
M(\boldsymbol{\alpha}',\boldsymbol{\alpha})
=
\begin{cases}
(-1)^{|\boldsymbol{\alpha}|-|\boldsymbol{\alpha}'|},
&
\boldsymbol{\alpha}'\preceq\boldsymbol{\alpha},
\\[4pt] 
0,
&
\text{otherwise}.
\end{cases}
\]
\end{lemma}

\begin{proof}
From \eqref{theta_jd_expand}, each term involving
$\theta_{j,\boldsymbol{\alpha}}$
contains the factor $\prod_{k=1}^K d_k^{\alpha_k}$. Therefore, after expanding
$\prod_{k=1}^K d_k^{\alpha_k}(1-d_k)^{1-\alpha_k}$, the monomials receiving contributions from $\theta_{j,\boldsymbol{\alpha}}$
are precisely those indexed by
$\boldsymbol{\alpha}'\succeq\boldsymbol{\alpha}$.
Equivalently, for a fixed interaction pattern $\boldsymbol{\alpha}$,
the coefficient of $\prod_{k=1}^K d_k^{\alpha_k}$ depends only on $\{\theta_{j,\boldsymbol{\alpha}'}:
\boldsymbol{\alpha}'\preceq\boldsymbol{\alpha}\}$.

Each factor $(1-d_k)$ contributes either $1$ or $-d_k$. When
$\boldsymbol{\alpha}'\preceq\boldsymbol{\alpha}$,
the interaction term $\prod_{k=1}^K d_k^{\alpha_k}$ is obtained from $\prod_{k=1}^K d_k^{\alpha'_k}(1-d_k)^{1-\alpha'_k}$ by selecting the factor $-d_k$ whenever $\alpha_k=1$ and $\alpha'_k=0$,
and selecting the factor $1$ whenever
$\alpha_k=\alpha'_k=0$.
Therefore, the corresponding coefficient is
\[
(-1)^{|\boldsymbol{\alpha}|-|\boldsymbol{\alpha}'|}.
\]
This proves the stated expression for $M$.
\end{proof}

In a PM-DINA model, only the coefficients corresponding to the constant term and the interaction term involving all required attributes are nonzero. Consequently, all other $2^K-2$ entries of $M\boldsymbol{\theta}_j$ are zero.
For aPM-CDMs defined by~\eqref{eq_aPM-CDM}, only the constant term and the first-order interaction terms may have nonzero coefficients. Equivalently, $(M\boldsymbol{\theta}_j)_{r}=0$ whenever the corresponding interaction feature involves more than one latent attribute.

Therefore, both PM-DINA and aPM-CDMs can be viewed as constrained subclasses of PM-GDINA. Because identifiability is preserved under these parameter restrictions, Theorem~\ref{main_theorem} immediately yields the following corollary.
\begin{corollary}
If an aPM-CDM, defined by the response probability parameterization~\eqref{eq_aPM-CDM}, satisfies one of the conditions in Theorem~\ref{main_theorem} or Theorem~\ref{relaxed_theorem}, then the model is locally identifiable.
\end{corollary}

\subsection{Examples of Real $Q$-Matrices}
In this subsection, we examine the $Q$-matrices in the data analyses of \citet{shang2021partial}. 
We summarize the results here and provide the corresponding $Q$-matrices in Appendix~X. Notice that we only propose sufficient identifiability conditions in this work.  Failure to satisfy these conditions does not necessarily imply that the model is not identifiable.

The fraction subtraction dataset \citep{tatsuoka1990toward} contains 20 items and 8 latent attributes. Only 3 items (6,8 and 9) require a single latent attribute, whereas all remaining items require at least two latent attributes. Consequently, the structure condition in Theorem~\ref{main_theorem} is not satisfied. Therefore, the identifiability result established in this paper does not directly apply to this dataset.

Another dataset considered by \citet{shang2021partial} is the 2003-2004 grammar section of the Examination for the Certificate of Proficiency in English (ECPE) test \citep{liu2009testing}. The dataset contains 28 items and 3 latent attributes: morphosyntactic form (\textit{Morph.}), cohesive form (\textit{Coh.}), and lexical form (\textit{Lex.}). 
For each latent attribute, there exist at least three pure items requiring only that attribute: 10, 13, and 14 for \textit{Morph.}; 8, 23, and 24 for \textit{Coh.}; and 4, 5, and 6 for \textit{Lex.}. Therefore, the structural condition in Theorem~\ref{main_theorem} is satisfied, and the local identifiability result established in this paper applies to this $Q$-matrix.


\section{Discussion}
In this paper, we developed the first identifiability results for partial-mastery cognitive diagnostic models (PM-CDMs).
Starting from a carefully chosen low-dimensional submodel, we used the Jacobian criterion to establish local identifiability and showed that this result can be extended to derive sufficient conditions for local identifiability of general PM-CDMs. Specifically, we established that the parametrization map of a PM-CDM is generically finite-to-one under suitable conditions. This result rules out continuous non-identifiability and implies that any ambiguity in parameter recovery is confined to a finite number of parameter points.

Our results contribute to the theoretical foundations of PM-CDMs and extend existing identifiability theory for standard CDMs, such as that of \citet{xu2017identifiability}. In particular, our sufficient conditions require that the $Q$-matrix contain two copies of the identity matrix and one additional pure item, whereas existing sufficient conditions for standard CDMs require only the former. Thus, identifiability in PM-CDMs is obtained under slightly stronger structural assumptions. Moreover, our results establish generic finite-to-one identifiability, whereas standard CDMs have been shown to satisfy global one-to-one identifiability under analogous conditions. This distinction reflects the additional flexibility introduced by PM-CDMs through continuous latent mastery variables.

The proposed conditions also have implications for assessment design. In particular, each latent attribute must be measured by at least two pure items, and one latent attribute must be measured by a third pure item. From a practical perspective, this highlights the importance of carefully designing item-attribute relationships when constructing assessments intended for PM-CDM analysis. 

Several important questions remain open. First, our results do not establish global one-to-one identifiability. Determining the exact number of observationally equivalent parameter points may be computationally challenging and, in some cases, infeasible. Second, the sufficient conditions developed here may be restrictive, particularly in assessments with a limited number of items or complex attribute structures. Furthermore, these conditions are unlikely to be necessary. Future research may seek to strengthen the present results by establishing global identifiability or by deriving weaker sufficient conditions for local identifiability. For example, proving that the key submodel in Step 1 is itself globally identifiable would immediately imply stronger identifiability results for the full PM-CDM. 
It would also be of interest to identify alternative local structures that could replace the current construction and thereby relax the requirement that the $Q$-matrix contain two identity sub-matrices and an additional pure item.

Overall, the present work provides a theoretical foundation for PM-CDMs by showing that identifiability can be established despite the substantially richer latent structure induced by partial mastery. More generally, in this paper, we illustrate how algebraic and analytic techniques can be combined to study identifiability in latent variable models with complex continuous-discrete structures. 


\section*{Online Supplementary Material}

Accompanying Macaulay2 code for the calculation of the rational map in Example 1 in Section~\ref{sec_theorem} is available at 
\href{https://osf.io/xsc9p/overview?view_only=65dad506447a4c73bf833ac9db480389}{https://osf.io/xsc9p/overview?view\_only=65dad506447a4c73bf833ac9db480389}.

\section*{Acknowledgement}

The work was funded by projects "Research of Excellence on Digital Technologies and Wellbeing CZ.02.01.01/00/22\_008/0004583'' and ``MSCA Fellowships on ICS CZ.02.01.01/00/22\_010/0012891'', which are co-financed by the European Union, by the Czech Science Foundation grant ``Complex analysis of educational measurement data to understand cognitive demands of assessment tasks'' number 25-16951S, and by the institutional support RVO 67985807.  
In preparing this manuscript, ChatGPT (OpenAI, 2026) was used for language editing and stylistical refinement. All content, including intellectual contributions, theoretical development, analysis, and interpretation, was developed by the authors. All AI-generated suggestions were reviewed, edited, and approved by the authors, who take full responsibility for the final content.


\bibliography{ref}

\newpage


\appendix
\section{The matrix $T_{I_5,\cdot}(Q,\Theta)$}\label{TI5}
The matrix  $T_{I_5,\cdot}(Q,\Theta)$ can be partitioned into four blocks of eight columns each.
\[
T_{I_5,\cdot}(Q,\Theta)=(T_1 \mid T_2 \mid T_3 \mid T_4),
\]
where
\begin{align*}
T_1=
\begin{blockarray}{cccccccc}
00000 & 10000 & 01000 & 11000 & 00100 & 10100 & 01100 & 11100\\
p^u_{00} & p^u_{10} & p^u_{10} & p^u_{20} & p^u_{10} & p^u_{20} & p^u_{20} & p^u_{30}\\
\begin{block}{(cccccccc)}
\theta_{10} & \theta_{11} & \theta_{10} & \theta_{11} & \theta_{10} & \theta_{11} & \theta_{10} & \theta_{11}\\
\theta_{20} & \theta_{20} & \theta_{21} & \theta_{21} & 
\theta_{20} & \theta_{20} & \theta_{21} & \theta_{21}\\
\theta_{30} & \theta_{30} & \theta_{30} & \theta_{30} & 
\theta_{31} & \theta_{31} & \theta_{31} & \theta_{31}\\
\theta_{40} & \theta_{40} & \theta_{40} & \theta_{40} & \theta_{40} & \theta_{40} & \theta_{40} & \theta_{40}\\
\theta_{50} & \theta_{50} & \theta_{50} & \theta_{50} & \theta_{50} & \theta_{50} & \theta_{50} & \theta_{50}\\
\end{block}
\end{blockarray},
\end{align*}
\begin{align*}
T_2=
\begin{blockarray}{cccccccc}
00010 & 10010 & 01010 & 11010 & 00110 & 10110 & 01110 & 11110\\
p^u_{01} & p^u_{11} & p^u_{11} & p^u_{21} & p^u_{11} & p^u_{21} & p^u_{21} & p^u_{31}\\
\begin{block}{(cccccccc)}
\theta_{10} & \theta_{11} & \theta_{10} & \theta_{11} & \theta_{10} & \theta_{11} & \theta_{10} & \theta_{11}\\
\theta_{20} & \theta_{20} & \theta_{21} & \theta_{21} & 
\theta_{20} & \theta_{20} & \theta_{21} & \theta_{21}\\
\theta_{30} & \theta_{30} & \theta_{30} & \theta_{30} & 
\theta_{31} & \theta_{31} & \theta_{31} & \theta_{31}\\
\theta_{41} & \theta_{41} & \theta_{41} & \theta_{41} & \theta_{41} & \theta_{41} & \theta_{41} & \theta_{41}\\
\theta_{50} & \theta_{50} & \theta_{50} & \theta_{50} & \theta_{50} & \theta_{50} & \theta_{50} & \theta_{50}\\
\end{block}
\end{blockarray},
\end{align*}
\begin{align*}
T_3=
\begin{blockarray}{cccccccc}
00001 & 10001 & 01001 & 11001 & 00101 & 10101 & 01101 & 11101\\
p^u_{01} & p^u_{11} & p^u_{11} & p^u_{21} & p^u_{11} & p^u_{21} & p^u_{21} & p^u_{31}\\
\begin{block}{(cccccccc)}
\theta_{10} & \theta_{11} & \theta_{10} & \theta_{11} & \theta_{10} & \theta_{11} & \theta_{10} & \theta_{11}\\
\theta_{20} & \theta_{20} & \theta_{21} & \theta_{21} & 
\theta_{20} & \theta_{20} & \theta_{21} & \theta_{21}\\
\theta_{30} & \theta_{30} & \theta_{30} & \theta_{30} & 
\theta_{31} & \theta_{31} & \theta_{31} & \theta_{31}\\
\theta_{40} & \theta_{40} & \theta_{40} & \theta_{40} & \theta_{40} & \theta_{40} & \theta_{40} & \theta_{40}\\
\theta_{51} & \theta_{51} & \theta_{51} & \theta_{51} & \theta_{51} & \theta_{51} & \theta_{51} & \theta_{51}\\
\end{block}
\end{blockarray},
\end{align*}
\begin{align*}
T_4=
\begin{blockarray}{cccccccc}
00011 & 10011 & 01011 & 11011 & 00111 & 10111 & 01111 & 11111\\
p^u_{02} & p^u_{12} & p^u_{12} & p^u_{22} & p^u_{12} & p^u_{22} & p^u_{22} & p^u_{32}\\
\begin{block}{(cccccccc)}
\theta_{10} & \theta_{11} & \theta_{10} & \theta_{11} & \theta_{10} & \theta_{11} & \theta_{10} & \theta_{11}\\
\theta_{20} & \theta_{20} & \theta_{21} & \theta_{21} & 
\theta_{20} & \theta_{20} & \theta_{21} & \theta_{21}\\
\theta_{30} & \theta_{30} & \theta_{30} & \theta_{30} & 
\theta_{31} & \theta_{31} & \theta_{31} & \theta_{31}\\
\theta_{41} & \theta_{41} & \theta_{41} & \theta_{41} & \theta_{41} & \theta_{41} & \theta_{41} & \theta_{41}\\
\theta_{51} & \theta_{51} & \theta_{51} & \theta_{51} & \theta_{51} & \theta_{51} & \theta_{51} & \theta_{51}\\
\end{block}
\end{blockarray}.
\end{align*}


\section{$Q$-matrices of the Data Examples in \citet{shang2021partial}}
\begin{table}[ht]
\centering
\label{tab:qmatrices}
\renewcommand{\arraystretch}{1.3} 

\begin{minipage}[t]{0.58\textwidth}
\centering
\textbf{(a) Fraction Subtraction}

\vspace{1em}
\begin{tabular}{c|cccccccc}
\toprule
Item & A1 & A2 & A3 & A4 & A5 & A6 & A7 & A8\\
\midrule
1  & 0 & 0 & 0 & 1 & 0 & 1 & 1 & 0 \\
2  & 0 & 0 & 0 & 1 & 0 & 0 & 1 & 0 \\
3  & 0 & 0 & 0 & 1 & 0 & 0 & 1 & 0 \\
4  & 0 & 1 & 1 & 0 & 1 & 0 & 1 & 0 \\
5  & 0 & 1 & 0 & 1 & 0 & 0 & 1 & 1 \\
6  & 0 & 0 & 0 & 0 & 0 & 0 & 1 & 0 \\
7  & 1 & 1 & 0 & 0 & 0 & 0 & 1 & 0 \\
8  & 0 & 0 & 0 & 0 & 0 & 0 & 1 & 0 \\
9  & 0 & 1 & 0 & 0 & 0 & 0 & 0 & 0 \\
10 & 0 & 1 & 0 & 0 & 1 & 0 & 1 & 1 \\
11 & 0 & 1 & 0 & 0 & 1 & 0 & 1 & 0 \\
12 & 0 & 0 & 0 & 0 & 0 & 0 & 1 & 1 \\
13 & 0 & 1 & 0 & 1 & 1 & 0 & 1 & 0 \\
14 & 0 & 1 & 0 & 0 & 0 & 0 & 1 & 0 \\
15 & 1 & 0 & 0 & 0 & 0 & 0 & 1 & 0 \\
16 & 0 & 1 & 0 & 0 & 0 & 0 & 1 & 0 \\
17 & 0 & 1 & 0 & 0 & 1 & 0 & 1 & 0 \\
18 & 0 & 1 & 0 & 0 & 1 & 1 & 1 & 0 \\
19 & 1 & 1 & 1 & 0 & 1 & 0 & 1 & 0 \\
20 & 0 & 1 & 1 & 0 & 1 & 0 & 1 & 0 \\
\bottomrule
\end{tabular}
\end{minipage}
\hfill
\begin{minipage}[t]{0.35\textwidth}
\centering
\textbf{(b) ECPE}

\vspace{1em}
\begin{tabular}{c|ccc}
\toprule
Item & Morph. & Coh. & Lex.\\
\midrule
1  & 1 & 1 & 0\\
2  & 0 & 1 & 0\\
3  & 1 & 0 & 1\\
4  & 0 & 0 & 1\\
5  & 0 & 0 & 1\\
6  & 0 & 0 & 1\\
7  & 1 & 0 & 1\\
8  & 0 & 1 & 0\\
9  & 0 & 0 & 1\\
10 & 1 & 0 & 0\\
11 & 1 & 0 & 1\\
12 & 1 & 0 & 1\\
13 & 1 & 0 & 0\\
14 & 1 & 0 & 0\\
15 & 0 & 0 & 1\\
16 & 1 & 0 & 1\\
17 & 0 & 1 & 1\\
18 & 0 & 0 & 1\\
19 & 0 & 0 & 1\\
20 & 1 & 0 & 1\\
21 & 1 & 0 & 1\\
22 & 0 & 0 & 1\\
23 & 0 & 1 & 0\\
24 & 0 & 1 & 0\\
25 & 1 & 0 & 0\\
26 & 0 & 0 & 1\\
27 & 1 & 0 & 0\\
28 & 0 & 0 & 1\\
\bottomrule
\end{tabular}
\end{minipage}

\end{table}





\end{document}